\documentclass[reqno,12pt]{amsart}
\usepackage{amsmath}
\usepackage{stmaryrd}
\usepackage{amsfonts}
\usepackage{amsthm}
\usepackage{amssymb}
\usepackage{graphicx,psfrag,epsfig}
\usepackage{color}
\usepackage{mathrsfs}
\usepackage{pdfsync}
\usepackage{cite}
\usepackage{times}
\usepackage{cancel}
\newtheorem{thm}{Theorem}[section]

\newtheorem{prop}[thm]{Proposition}

\newtheorem{rem}[thm]{Remark}

\DeclareMathAlphabet{\mathpzc}{OT1}{pzc}{m}{it}

\numberwithin{equation}{section}

\newcommand{\R}{\mathbb{R}}

\newcommand{\ve}{\varepsilon}

\newcommand{\rd}{\mathrm{d}}
\newcommand{\sh}{\sigma_*}

\newcommand{\dhr}{\mathrel{\lhook\joinrel\relbar\kern-.8ex\joinrel\lhook\joinrel\rightarrow}} 
\title[Heterogeneous dielectric properties in MEMS models]
{Heterogeneous dielectric properties in MEMS models}
\author{Philippe Lauren\c{c}ot}
\thanks{ Partially supported by the French-German PROCOPE project~30718ZG}
\address{Institut de Math\'ematiques de Toulouse, UMR~5219, Universit\'e de Toulouse, CNRS \\ F--31062 Toulouse Cedex 9, France}
\email{laurenco@math.univ-toulouse.fr}

\author{Christoph Walker}
\address{Leibniz Universit\"at Hannover\\ Institut f\" ur Angewandte Mathematik \\ Welfengarten 1 \\ D--30167 Hannover\\ Germany}
\email{walker@ifam.uni-hannover.de}

\date{\today}

\begin{document}

\begin{abstract}
An idealized electrostatically actuated microelectromechanical system (MEMS)  involving an elastic plate with a heterogeneous dielectric material is considered. Starting from the electrostatic and mechanical energies, the  governing evolution equations for the electrostatic potential and the plate deflection are derived from the corresponding energy balance. This leads to a free boundary transmission problem due to a jump of the  dielectric permittivity across the interface separating elastic plate and free space. 
Reduced models retaining the influence of the heterogeneity of the elastic plate under suitable assumptions are obtained when  either the elastic's plate thickness or the aspect ratio of the device vanishes.
\end{abstract}


\maketitle

\section{Introduction}

Microelectromechanical systems (MEMS) are important parts in modern technology \cite{PeB03, You11}. In an idealized setting, a electrostatically actuated MEMS device consists of a rigid conducting ground plate above which an elastic plate, coated with a thin dielectric layer, is suspended. Holding the two plates at different electrostatic potentials induces a Coulomb force across the device deforming the elastic plate, thereby modifying the shape of the device and transforming electrostatic energy into mechanical energy. The modeling of such an idealized MEMS involves in general the vertical deflection $u$ of the elastic plate and the electrostatic potential $\psi$ in the device. More specifically, let us consider a rigid ground plate of shape $D\subset \mathbb{R}^2$ and an elastic plate with the same shape $D$ at rest and uniform thickness $d>0$ and being made of a possibly non-uniform dielectric material. Denoting the vertical deflection of the bottom of the elastic plate at a point $x=(x_1,x_2)\in  D$ by $u=u(x)$, the deformed elastic plate is given by 
$$
 {\Omega}_2(u):=\left\{(x,z)\in D\times \mathbb{R}\,;\, u(x)<  z <  u(x)+d\right\}\,,
$$
so that its top surface is located at height $z= u(x)+d$, $x\in D$. The elastic plate being suspended at its boundary has zero deflection there, that is, $u(x)=0$ for $x\in\partial  D$. As for the rigid ground plate of shape $D$, it is located at $z=-H$ and it is held at zero potential while the elastic plate is held at a constant potential $V>0$.

The region between the two plates is described by
$$
{\Omega}_1(u):=\left\{(x, z)\in D\times\mathbb{R}\,;\, -H< z< {u}(x)\right\}\,,
$$
and separated from the elastic plate by the interface 
$$
\Sigma(u):=\{(x,z)\in D\times\mathbb{R}\,;\, z=  u(x)\}\,.
$$
Due to the above described geometry of the MEMS under study, the electrostatic potential $\psi_u$ in the device is defined in a non-homogeneous medium which is endowed with the following properties: the medium or vacuum filling the region $\Omega_1(u)$ between the plates is assumed to have constant permittivity $\sigma_1>0$ while the dielectric properties are allowed to vary across the elastic plate material, a fact which is reflected by a non-constant permittivity $\sigma_2 = \sigma_2(x)$ (though independent of the vertical direction at this stage for simplicity). Introducing the electrostatic potentials between the plates $\psi_{u,1}= \psi_{u,1}(x,z)$, $(x,z)\in \Omega_1(u)$, and within the elastic plate $\psi_{u,2}= \psi_{u,2}(x,z)$, $(x,z)\in \Omega_2(u)$, the electrostatic potential $\psi_u$ and the permittivity $\sigma$ in the device are given by
$$
\psi_u:=\left\{\begin{array}{ll} \psi_{u,1} & \text{in}\ \Omega_1(u)\,,\\
\psi_{u,2} & \text{in}\ \Omega_2(u)\,,
\end{array}\right.\qquad 
\sigma:=\left\{\begin{array}{ll} \sigma_1 & \text{in}\ \Omega_1(u)\,,\\
\sigma_2 & \text{in}\ \Omega_2(u)\,.
\end{array}\right.
$$
Since the electrostatic potential $\psi_u$ is defined in the domain 
$$
\Omega(u):=\left\{(x,z)\in D\times\mathbb{R}\,;\, -H<  z <  u(x)+d\right\}= \Omega_1( {u})\cup  \Omega_2(u)\cup \Sigma(u)\,,
$$
which varies according to the deformation $u$, the definition of the former is obviously strongly sensitive to the geometry of the latter. In particular,  the permittivity $\sigma$ features a jump at the interface $\Sigma(u)$ that changes its location with $u$. Moreover, the region $\Omega_1(u)$ between the two plates is connected only when the two plates remain separate, that is, as long as $u(x)>-H$ for all $x\in D$. This corresponds to a stable operating condition of the MEMS device. However, it is expected that there is a critical threshold value for the applied voltage difference $V$ above which the restoring elastic forces can no longer balance the attractive electrostatic forces and the top plate ``pulls in'', that is, sticks onto the ground plate. This \textit{touchdown} phenomenon manifests itself in a situation in which $u(x)=-H$ for some $x\in D$.  When the thickness of the elastic plate is neglected, a touchdown leads to a breakdown of the model (or, alternatively, in a singularity from a mathematical viewpoint) \cite{BGP00, FMPS06, GPW05, Pe02, PeB03}. This, however, need not be the case for plates with positive thickness as it then corresponds to a \textit{zipped state} with the elastic plate lying directly on the ground plate \cite{GB01}.

\bigskip

The first purpose of this work is to derive a mathematical model for the dynamics of the above described MEMS device when the thickness of the elastic plate and hence its dielectric properties are explicitly taken into account. The approach adopted herein is in accordance with \cite{AmEtal, FMCCS05, Jac62, CLWZ13} and prescribes the dynamics of $u$ as the gradient flow of the \textit{total energy} $E(u)=E_m(u)+E_e(u)$. It includes the \textit{mechanical energy}
$$
E_m ( {u}):=\frac{B}{2}\int_{ D}  \left\vert\Delta   u\right\vert^2\,\rd x
+\frac{T}{2}\int_{  D}  \left\vert \nabla  u\right\vert^2\, \rd x\,,
$$
where the first term accounts for plate bending with coefficient $B\ge 0$, and the second term accounts for stretching with coefficient $T\ge 0$. For simplicity we restrict ourselves to vertical deflections and consider only linear bending by neglecting curvature effects from the outset. Moreover, we refrain at this point from modeling the above mentioned zipped states and refer to the subsequent sections for this issue. The \textit{electrostatic energy} is given by
$$
E_e(u):=-\frac{1}{2}\int_{\Omega(u)} \sigma \vert\nabla \psi_u\vert^2\,\rd (x,z)\,,
$$
where the electrostatic potential $\psi_u$ is the maximizer of the Dirichlet integral
$$
-\frac{1}{2}\int_{\Omega(u)} \sigma \vert\nabla \vartheta\vert^2\,\rd (x,z)
$$
among functions $\vartheta\in H^1(\Omega(u))$ satisfying appropriate boundary conditions (see Section~\ref{Sec2}) on $\partial\Omega(u)$. The electrostatic energy $E_e(u)$ then clearly depends on the deflection $u$ not only through the domain of integration $\Omega(u)$ but also through the implicit dependence of the electrostatic potential $\psi_u$ on $u$. Since the derivation of the corresponding mathematical model is based on the energy balance, it requires the computation of the first variation $\delta_u E(u) = \delta_u E_m(u) + \delta_u E_e(u)$ of the total energy $E(u)$ with respect to $u$. The computation of $\delta_u E_e(u)$ turns out to be quite involved as we shall see in Section~\ref{Sec2} below. It yields the electrostatic force $F_e(u)$ exerted on the elastic plate in the form
\begin{subequations}\label{omarhakim}
\begin{equation}
F_e(u):=\delta_u E_e(u)=F_{e,1}(u)+F_{e,2}(u) \label{omarhakim0}
\end{equation}
where
\begin{equation}\label{omarhakim1}
F_{e,1}(u)(x) := - \frac{1}{2}\,\frac{\sigma_1-\sigma_2(x)}{1+ \vert\nabla u(x)\vert^2} \tilde{F}_{e,1}(u)(x)\,,
\end{equation}
with
\begin{align*}
\tilde{F}_{e,1}(u)(x) := & \Big\vert \partial_z\psi_{u,2}(x,u(x))\nabla u(x) +\nabla'\psi_{u,2}(x,u(x))\Big\vert^2 \\
&  \quad + \Big(\nabla'\psi_{u,2}(x,u(x)) \cdot \nabla^\perp u(x)\Big)^2 \\
& \quad +\frac{\sigma_2(x)}{\sigma_1}\Big( \partial_z\psi_{u,2}(x,u(x)) -\nabla u(x)\cdot\nabla'\psi_{u,2}(x,u(x))\Big)^2
\end{align*} 
and
\begin{equation}\label{omarhakim2}
\begin{split}
&F_{e,2}(u)(x):= \frac{1}{2}\sigma_2(x)  \left\vert \nabla\psi_{u,2}(x,u(x)+d)\right\vert^2\,,
\end{split}
\end{equation}
\end{subequations}
where $\nabla':=(\partial_{x_1},\partial_{x_2})$, $\nabla^\perp := (\partial_{x_2},-\partial_{x_1})$.
Let us point out here again that, besides the complexity of the formula \eqref{omarhakim} giving the electrostatic force in terms of $u$ and $\psi_u$, the electrostatic potential $\psi_u$ itself depends in an implicit and intricate way on the deflection $u$ as it solves a transmission elliptic boundary value problem on a domain which varies with respect to $u$. The precise equations are stated in the next section, see \eqref{psi}.
 
If the thickness of the plate is neglected, that is, if $d=0$, then $\psi_{u}=\psi_{u,1}$ and the corresponding electrostatic force reduces to
\begin{equation}\label{omarhakim25}
F_{e}(u)(x):= \frac{1}{2}\sigma_1  \left\vert \nabla\psi_{u}(x,u(x))\right\vert^2
\end{equation}
as already derived in \cite{FMCCS05, LW14b}.

\medskip

A somewhat different approach is pursued in \cite{Pe02} where the electrostatic force $F_e(u)$ is a priori  assumed to be proportional to the square of the gradient trace of the electrostatic potential on the elastic plate. More precisely, if $d>0$, then the electrostatic force is taken to be
\begin{equation}\label{FPelesko}
F_e(u):= F_{e,2}(u) \,,
\end{equation}
where $F_{e,2}(u)$ is defined in \eqref{omarhakim2}, and if $d=0$, then it is given by \eqref{omarhakim25}. Interestingly, both approaches give rise to the same electrostatic force \eqref{omarhakim25}  when the thickness $d$ of the elastic plate is neglected and taken to be equal to zero. However, when the thickness is positive, the electrostatic force \eqref{omarhakim} includes additional terms compared to \eqref{omarhakim25}, which are gathered in $F_{e,1}(u)$ and stem from the discontinuity of the permittivity across the interface $\Sigma(u)$. Observe that $F_{e,1}(u)$ is nonnegative in the physically relevant situation where $\sigma_2\ge \sigma_1$.

\medskip

Having the electrostatic force $F_e(u)$ from \eqref{omarhakim} at hand we are in a position to write the force balance which yields the evolution of the deflection $u=u(t,x)$ in the form 
\begin{equation}\label{theedge}
\begin{split}
&\alpha_0 \partial_t^2u+r\partial_t u +  B  \Delta^2   u - T \Delta u = - F_e(u)\,,\qquad x\in D\,,\quad t>0\,. 
\end{split}
\end{equation}
Here, $\alpha_0 \partial_t^2u$ accounts for inertia forces, $r\partial_t u$ is a damping force, and 
\begin{equation}\label{nirvana}
B \Delta^2   u - T \Delta u  = \delta_u E_m(u)\,.
\end{equation}
Consequently, the evolution of the elastic plate deflection $u$ is given by a semilinear damped wave equation \eqref{theedge} with a nonlocal source term involving, in particular, the square of the trace of the gradient of the electrostatic potential, the latter being a solution to an elliptic transmission problem (see \eqref{psi} below) on a domain depending on $u$. Thus, in addition to a complicated expression for the electrostatic force $F_e(u)$, there is a strong coupling between the deflection $u$ and the electrostatic potential $\psi_u$. To get a better insight into the dynamics it is therefore of utmost importance to derive reduced models, which are more tractable from an analytical point of view. A first step in this direction is to investigate the limiting behavior of the model as the plate thickness $d$ vanishes. In Section~\ref{Sec3.1} we first consider the case where the dielectric permittivity $\sigma_2$ is of order~$1$ with respect to $d$. Amazingly, no influence of the permittivity $\sigma_2$ is retained in this limit. The model we end up with is just \eqref{theedge} with electrostatic force given by \eqref{omarhakim25}. This is in sharp contrast to the second situation that we consider in Section~\ref{RC} in which $\sigma_2$ is of order~$d$. If $\sigma_2=d\sh$, we find that the electrostatic force is then given by
\begin{equation*}
\begin{split}
F_e(u)(x)=&\frac{\sigma_1}{2} |\nabla\psi_u(x,u(x))|^2 - \mathrm{div}\left( \sh(x) (\psi_u(x,u(x))-V)^2 \nabla u (x)\right) \\
& + \sh(x) (1+|\nabla u(x)|^2) (\psi_u(x,u(x))-V) \partial_z\psi_u(x,u(x)) \, . 
\end{split}
\end{equation*}
In addition, the boundary value problem for the electrostatic potential $\psi_u$ is of a different nature (see \eqref{PhRe2} below for details). In both cases the models obtained in the limit $d\to 0$ are still rather complex. 

Section~\ref{Sec4} is then devoted to the classical vanishing aspect ratio limit which amounts to let  $H/\mathrm{diam}(D)$ go to zero \cite{Pe02}. This procedure allows one to express $\psi_u$ as well as the electrostatic force $F_e(u)$ explicitly in terms of $u$. The vanishing aspect ratio model we thus obtain reads (after a suitable rescaling)
\begin{equation}\label{omarhakim27}
\begin{split}
&\gamma^2 \partial_t^2u+\partial_t u +  \beta \Delta^2   u - \tau \Delta u = - \frac{\lambda}{2(1+u+\sigma_2^{-1})^2} \,,\qquad x\in D\,,\quad t>0\,.
\end{split}
\end{equation}
In this situation, the pull-in instability occurs if $u$ reaches the value $-1$ which is the vertical position of the ground plate in rescaled variables. It is worth pointing out that this instability does not correspond to a singularity in the electrostatic force when $\sigma_2^{-1}>0$. This is consistent with the original model \eqref{theedge}, where the positive thickness of the plate prevents the occurrence of a singularity. Touchdown singularities can only occur  at points where the elastic plate is a perfect conductor meaning that $\sigma_2(x)^{-1}=0$. This is contrary to the widely used model derived in \cite{Pe02}, where touchdown singularities may occur only at dielectric points.
 
\section{Model}\label{Sec2}

In this section we provide a detailed derivation of the equation~\eqref{theedge} with electrostatic force given in \eqref{omarhakim} and first recall the description of the device. We assume that the elastic plate is made of a dielectric material and has a uniform thickness $d>0$. Its shape is an open bounded domain $D \subset \R^2$ with sufficiently smooth boundary.\footnote{$D$ can also be an interval in $\R$ in the following and then $x$ is a scalar.} The rigid ground plate is located at $  z=-H$, while the elastic plate is 
$$
 {\Omega}_2(u):=\left\{(x,z)\in D\times \mathbb{R}\,;\, u(x)<  z <  u(x)+d\right\}\,,
$$
where
$u=u(x)$ denotes  the deflection of the bottom of the elastic plate at a point $x=(x_1,x_2)\in  D$. For consistency of the model, we presuppose that $u(x)>-H$ for $x\in D$. Since the elastic plate is suspended above the ground plate, there is zero deflection 
$$
u(x)=0\,,\quad x\in\partial  D\,,
$$
at the boundary of $D$. If the plate  is assumed to be clamped, then one requires in addition a vanishing normal derivative
$$
\partial_\nu  u(x)=0\,,\quad x\in\partial  D\,,
$$
with $\nu$ denoting the outward unit normal on $\partial  D$. The interface 
$$
\Sigma(u):=\{(x,z)\in D\times \mathbb{R}\,;\, z=  u(x)\}
$$
separates the elastic plate ${\Omega}_2(u)$ from the region ${\Omega}_1(u)$ between the two plates, given by
$$
{\Omega}_1(u):=\left\{(x, z)\in D\times \mathbb{R} \,;\, -H< z< {u}(x)\right\}\,.
$$
We set
$$
 {\Omega}( {u}):=\left\{(x,z)\in D\times \mathbb{R} \,;\, -H<  z <  u(x)+d\right\}= {\Omega}_1( {u})\cup  {\Omega}_2( {u})\cup \Sigma(u)\,
$$
and let 
\begin{equation}\label{normal}
{\bf n}_{ \Sigma(u)}(x)=\frac{(-\nabla u(x), 1)}{\sqrt{1+\vert\nabla u(x)\vert^2}}\ , \quad x\in   D\ , 
\end{equation}
denote the unit normal on the interface $ \Sigma(u)$ pointing into $ \Omega_2(u)$.
 
The top surface of the elastic plate is kept at a constant positive voltage value $V$ while  the ground plate is kept at zero voltage. Let $ \sigma_1$ be the constant permittivity of the medium (or vacuum) filling the region between the plates,  and let $ \sigma_2$ be the permittivity of the plate material which we do not assume to be homogeneous in this paper. In fact, we assume that the dependence on the vertical direction (if any) involves the relative position of the plate with respect to its deflection $u$, that is, 
\begin{equation}
\sigma_2(x,z) := \sh(x,z-u(x))\ , \quad (x,z)\in \Omega_2(u)\ , \label{asterix}
\end{equation}
where $\sh$ is a function defined on $\bar D\times [0,d]$. 
Given a pair of real-valued functions $(\vartheta_1,\vartheta_2)$ with $\vartheta_j$ defined on $\Omega_j(u)$, $j=1,2$, we put
$$
\vartheta:=\left\{\begin{array}{ll} \vartheta_1 & \text{in}\ \Omega_1(u)\\
\vartheta_2 & \text{in}\ \Omega_2(u)
\end{array}\right.\,,\qquad 
\sigma:=\left\{\begin{array}{ll} \sigma_1 & \text{in}\ \Omega_1(u)\,,\\
\sigma_2 & \text{in}\ \Omega_2(u)\,.
\end{array}\right.
$$
Also, for a function $f=f(x,z)$ we slightly abuse notation by writing $f[\cdot]$ whenever the $x$-variable is omitted, that is, for example $f[u]=f(x,u(x))$ and $\sh[z-u]=\sh(x,z-u(x))$. 

To define the boundary conditions for the electrostatic potential we fix a smooth function
\begin{subequations}\label{bobbybrown}
\begin{equation}\label{bobbybrown1}
h: \bar D\times (-H,\infty)\times (-H,\infty)\rightarrow \R\;\text{ with }\; h(x,w,w)=V\,,\  h(x,-H,w)=0
\end{equation}
for $(x,w)\in \bar D\times (-H,\infty)$ and define 
\begin{equation}\label{bobbybrown2}
h_v(x,z):=h(x,z,v(x))\,,\quad (x,z)\in \bar D\times (-H,\infty)\,,
\end{equation}
\end{subequations}
for a given function $v:\bar D\rightarrow (-H,\infty)$.

\subsection{Electrostatic Potential}

We now introduce the functional
$$
\mathcal{E}(u,\vartheta):=-\frac{1}{2}\int_{ \Omega(u)} \sigma \vert\nabla \vartheta\vert^2\,\rd (x,z)\,
$$
for a sufficiently smooth deflection $u:\bar D\rightarrow (-H,\infty)$ and $\vartheta\in h_{u+d}+H_{0}^1(\Omega(u))$, where $H_{0}^1(\Omega(u))$ denotes the subspace of the Sobolev space $H^1(\Omega(u))$ consisting of those functions with zero trace on the boundary of $\Omega(u)$. The electrostatic potential $\psi_u$ in $\Omega(u)$ is the maximizer of this functional with respect to $\vartheta\in h_{u+d}+H_{0}^1(\Omega(u))$. Then the \textit{electrostatic energy} of the device is given by
\begin{equation}\label{sugarcane}
E_e(u):=\mathcal{E}(u,\psi_u)=-\frac{1}{2}\int_{ \Omega(u)} \sigma \vert\nabla \psi_u\vert^2\,\rd (x,z)\,.
\end{equation}
To derive the equations for the electrostatic potential $\psi_u$ depending on the given deflection $u$, we use the fact that it is a critical point  of $\mathcal{E}(u,\vartheta)$ with respect to $\vartheta$. Letting $\phi\in C_c^1(\Omega(u))$, we obtain from Gauss' theorem
\begin{equation*}
\begin{split}
\partial_\vartheta \mathcal{E}(u,\vartheta)\phi &=-\int_{\Omega(u)} \sigma \nabla\vartheta\cdot\nabla\phi\,\rd (x,z)\\
&= \int_{\Omega(u)} \phi\, \mathrm{div}\left(\sigma\nabla\vartheta\right)\,\rd (x,z) -\int_{\Sigma(u)}
 \phi\, \llbracket \sigma\nabla \vartheta \rrbracket \cdot {\bf n}_{ \Sigma(u)}\,\rd S\,,
\end{split}
\end{equation*}
where $ \llbracket f \rrbracket:=f_1-f_2$ stands for the jump across the interface $\Sigma(u)$ of a function $f$ defined in $\Omega_1(u)\cup \Omega_2(u)$.
Consequently, the electrostatic potential $\psi_u$ for a given deflection $u$ satisfies
\begin{subequations}\label{psi}
\begin{equation}\label{psi00}
\mathrm{div}\left( \sigma\nabla \psi_u\right)=0 \quad\text{in}\quad  {\Omega}({u})
\end{equation}
with \textit{transmission conditions} on the interface $\Sigma(u)$ 
\begin{equation}\label{TM1}
 \llbracket \psi_u\rrbracket = \llbracket \sigma\nabla \psi_u \rrbracket \cdot {\bf n}_{ \Sigma(u)}= 0 \quad\text{on}\quad  \Sigma(  u)\,,
\end{equation}
 along with the boundary conditions
\begin{equation}
 \psi_u=h_{u+d}\,,\quad  (x,z)\in \partial\Omega(u)\,,\label{psibc1}
\end{equation}
\end{subequations}
with $h_{u+d}$ being defined in \eqref{bobbybrown}.

\subsection{Electrostatic Force}\label{Sec2.2}

We now derive the electrostatic force exerted on the elastic plate by computing the first variation $\delta_u E_e(u)$.  Let a (smooth) deflection $u:\bar D\rightarrow \R$ be fixed, vanishing on $\partial D$ with $u>-H$ in $D$. Let $\psi_u$ be the corresponding solution to \eqref{psi} in $\Omega(u)$ and note that  $\psi_u$ depends non-locally on the deflection $u$. Let $v\in C_0^\infty(D)$ and put $u_s:=u+sv$ for $s\in (-\sigma_0,\sigma_0)$, where $\sigma_0$ is chosen small enough so that $u_s>-H$ in $D$ for all $s\in (-\sigma_0,\sigma_0)$. The goal is to compute
$$
\delta_u E_e(u)v=\frac{\mathrm{d}}{\mathrm{d}s} E_e(u+sv)\vert_{s=0}\,.
$$
Since $u$ is fixed throughout this section, we simply write $\psi$, $\Omega$, $\Omega_1$, $\Omega_2$, and $\Sigma$ instead of $\psi_u$, $\Omega(u)$, $\Omega_1(u)$, $\Omega_2(u)$, and $\Sigma(u)$, respectively. For $s\in (-\sigma_0,\sigma_0)$, we introduce the transformation $\Phi(s):=(\Phi_1(s),\Phi_2(s))$ with
\begin{align*}
\Phi_1(s)(x,z) & := \left(x,z+sv(x)\frac{H+z}{H+u(x)}\right)\,,&& (x,z)\in \Omega_1\ , \\
\Phi_2(s)(x,z) & :=\left(x,z+sv(x)\right)\,,&& (x,z)\in \Omega_2\,.
\end{align*}
Note that
$$
\Omega_\ell(u_s)=\Phi_\ell(s)(\Omega_\ell)\,,\quad \ell=1,2\ ,
$$
and 
$$
\mathrm{det}(\nabla\Phi_1(s))= 1 + \frac{sv}{H+u}>0\ , \qquad \mathrm{det}(\nabla\Phi_2(s))=1\ .
$$ 
Moreover, 
\begin{align}
\partial_s\Phi_1(0)(x,z) & =\left(0,\frac{v(x)(H+z)}{H+u(x)}\right)\,,&& (x,z)\in\Omega_1\,, \label{Phi1s} \\
\partial_s\Phi_2(0)(x,z) & =\left(0,v(x)\right)\,,&& (x,z)\in\Omega_2\,. \label{Phi2s}
\end{align}
Let now $\psi(s)$ be the solution to \eqref{psi} in $\Omega(u_s)$, that is,
\begin{subequations}\label{psiss}
\begin{align}
\mathrm{div}\left( \sigma(s) \nabla {\psi(s)}\right)&=0 &&\text{in}\quad  \Omega(u_s)\,,\label{psis}\\
 \llbracket \psi(s)\rrbracket=\llbracket \sigma(s) \nabla \psi(s) \rrbracket \cdot {\bf n}_{ \Sigma(u_s)}&= 0 &&\text{on}\quad  \Sigma(u_s)\,, \label{TM2s}\\
 \psi(s)&=h_{u_s+d}\,,&&  (x,z)\in \partial\Omega(u_s)\,,\label{psibc1s}
\end{align}
\end{subequations}
where
$$
\sigma(s)(x,z) := \left\{ \begin{array}{lcl}
\sigma_1 & \text{ for } & (x,z)\in \Omega_{1}(u_s)\,, \\
 & &\\
\sh(x,z-u_s(x)) & \text{ for } & (x,z)\in \Omega_{2}(u_s)\,.
\end{array}
\right. 
$$
Then $\psi(0)=\psi$ and $\sigma(0)=\sigma$. To compute the derivative with respect to $s$ of
$$
E_e(u_s)=-\frac{1}{2}\int_{ \Omega(u_s)} \sigma(s) \vert\nabla \psi(s)\vert^2\,\rd (x,z)
$$
we use the Reynolds transport theorem (e.g. see \cite[XII.Theorem 2.11]{AEIII}) and obtain
\begin{align*}
\frac{\mathrm{d}}{\mathrm{d} s} E_e(u_s)\vert_{s=0} = & - \int_{\Omega} \left[  \sigma\nabla\psi\cdot\nabla\partial_s\psi(0) + \mathrm{div}\left(\frac{\sigma}{2}\vert \nabla\psi\vert^2\partial_s\Phi(0)\right)\right]\,\rd (x,z) \\
& \qquad - \frac{1}{2} \int_\Omega \partial_s \sigma(0) |\nabla\psi|^2 \,\rd (x,z) \,.
\end{align*}
From Gauss' theorem, \eqref{psi}, and the definition of $s\mapsto \sigma(s)$ it follows that
\begin{equation}\label{DE1}
\begin{split}
\frac{\mathrm{d}}{\mathrm{d} s} E_e(u_s)\vert_{s=0}=&-\int_{\partial\Omega} \sigma \left(\partial_s\psi(0)\nabla\psi+\frac{1}{2}\vert \nabla\psi\vert^2\partial_s\Phi(0)\right)\cdot {\bf n}_{\partial\Omega}\,\rd S\\
&- \int_{\Sigma}  \left\llbracket \sigma\partial_s\psi(0)\nabla\psi+\frac{\sigma}{2}\vert \nabla\psi\vert^2\partial_s\Phi(0)\right\rrbracket\cdot {\bf n}_{\Sigma}\,\rd S \\
& + \frac{1}{2} \int_{\Omega_2} v \partial_z \sh[z-u] |\nabla\psi_2|^2\, \rd (x,z) \,.
\end{split}
\end{equation}
Note that \eqref{bobbybrown} and \eqref{psibc1s} entail that $\partial_s\psi(0)=0$ on the parts $\partial D\times [-H<z<u+d]$ and $D\times\{-H\}$ of the boundary $\partial\Omega$. Also,  $\partial_s\Phi(0)=0$ on $D\times\{-H\}$ and on $\partial D\times [-H<z<u+d]$ (as $v$ vanishes on $\partial D$). Thus, the corresponding boundary integrals vanish and \eqref{DE1} reduces to
\begin{equation}\label{DE2}
\begin{split}
\frac{\mathrm{d}}{\mathrm{d} s} E_e(u_s)\vert_{s=0}=&-\int_{[z=u+d]} \sigma \left(\partial_s\psi(0)\nabla\psi+\frac{1}{2}\vert \nabla\psi\vert^2\partial_s\Phi(0)\right)\cdot {\bf n}_\Sigma\,\rd S\\
&- \int_{\Sigma}  \left\llbracket \sigma\partial_s\psi(0)\nabla\psi+\frac{\sigma}{2}\vert \nabla\psi\vert^2\partial_s\Phi(0)\right\rrbracket\cdot {\bf n}_\Sigma\,\rd S\\
& + \frac{1}{2} \int_{\Omega_2} v \partial_z \sh[z-u] |\nabla\psi_2|^2\, \rd (x,z) \,.
\end{split}
\end{equation}
We now compute the first integral on the right-hand side of \eqref{DE2}. To that end, we recall that
$$
\psi(s)\big(x,u_s(x)+d\big)=V\,,\quad  x\in D\,,
$$
according to \eqref{bobbybrown} and \eqref{psibc1s}, and thus 
\begin{equation}
\partial_s\psi(0)\big(x,u(x)+d\big)=-\partial_z\psi\big(x,u(x)+d\big)v(x)\,,\quad x\in D\,,\label{x1}
\end{equation}
and
\begin{equation}
\nabla'\psi_2(x,u(x)+d) = - \partial_z \psi_2(x,u(x)+d) \nabla u(x)\,,\quad x\in D\,. \label{x1.1}
\end{equation}
From \eqref{Phi2s}, \eqref{x1}, and \eqref{x1.1} we obtain
\begin{equation}\label{x2}
\begin{split}
\int_{[z=u+d]} & \sigma \left(\partial_s\psi(0)\nabla\psi+\frac{1}{2}\vert \nabla\psi\vert^2\partial_s\Phi(0)\right)\cdot {\bf n}_\Sigma\,\rd S\\
&= \int_{D}\sh [d] \left( \frac{1}{2} \vert \nabla \psi_2[u+d]\vert^2 +\partial_z\psi_2 [u+d] \nabla u \cdot \nabla'\psi_2 [u+d]\right) v\,\rd x  \\
& \qquad - \int_{D}\sh [d] \vert \partial_z\psi_2 [u+d]\vert^2 v\,\rd x  \\
& = -\frac{1}{2}\int_{D}\sh [d] \vert \nabla\psi_2 [u+d]\vert^2 v\,\rd x\,.
\end{split}
\end{equation}
 
We next consider the second term on the right-hand side of \eqref{DE2} which involves an integral over $\Sigma$. First note that \eqref{TM1} implies
$$
\left\llbracket \sigma\partial_s\psi(0)\nabla\psi \right\rrbracket\cdot {\bf n}_\Sigma =\frac{1}{2} \left\llbracket \partial_s\psi(0) \right\rrbracket  \big(\sigma_1\nabla\psi_1 +\sigma_2\nabla\psi_2\big) \cdot {\bf n}_\Sigma \quad\text{on}\ \ \Sigma\,,
$$
while differentiating \eqref{TM2s} with respect to $s$ gives
$$
\left\llbracket \partial_s\psi(0) \right\rrbracket=-\left\llbracket \partial_z\psi \right\rrbracket v \quad\text{on}\ \ \Sigma\,.
$$
Hence,
\begin{equation}\label{x3}
\begin{split}
\left\llbracket \sigma\partial_s\psi(0)\nabla\psi \right\rrbracket\cdot {\bf n}_\Sigma &=-\frac{v}{2}  \left\llbracket \partial_z\psi \right\rrbracket \big(\sigma_1\nabla\psi_1 +\sigma_2\nabla\psi_2\big) \cdot {\bf n}_\Sigma \\
&= \frac{v}{2}\frac{ \left\llbracket \partial_z\psi \right\rrbracket}{\sqrt{1+\vert\nabla u\vert^2}} \big(\sigma_1\nabla'\psi_1+\sigma_2\nabla'\psi_2\big)\cdot\nabla u \\
& \qquad - \frac{v}{2}\frac{ \left\llbracket \partial_z\psi \right\rrbracket}{\sqrt{1+\vert\nabla u\vert^2}} \big(\sigma_1\partial_z\psi_1+\sigma_2\partial_z\psi_2\big)
\end{split} 
\end{equation}
on $\Sigma$.
Now, since \eqref{TM1}, after differentiating it with respect to $x$, implies
\begin{equation}\label{TM1ax}
-\left\llbracket \partial_z\psi\right\rrbracket\nabla u=\left\llbracket\nabla'\psi \right\rrbracket  \quad\text{on}\ \ \Sigma\,,
\end{equation}
while \eqref{Phi1s}-\eqref{Phi2s} entail
$$
\left\llbracket \sigma\vert \nabla\psi\vert^2\partial_s\Phi(0)\right\rrbracket\cdot {\bf n}_\Sigma =\frac{v}{\sqrt{1+\vert\nabla u\vert^2}} \left\llbracket \sigma\vert \nabla\psi\vert^2\right\rrbracket \quad\text{on}\ \ \Sigma\,,
$$ 
it follows from \eqref{x3} that
\begin{equation*}
\begin{split}
&\left\llbracket \sigma\partial_s\psi(0)\nabla\psi+\frac{\sigma}{2}\vert \nabla\psi\vert^2\partial_s\Phi(0)\right\rrbracket\cdot {\bf n}_\Sigma\\
&\qquad =\frac{1}{2}\frac{v}{\sqrt{1+\vert\nabla u\vert^2}}\left\{-\big(\sigma_1\nabla'\psi_1+\sigma_2\nabla'\psi_2\big) \llbracket\nabla'\psi\rrbracket   - \big(\sigma_1\partial_z\psi_1+\sigma_2\partial_z\psi_2\big) \llbracket\partial_z\psi\rrbracket \right\} \\
&\qquad\qquad + \frac{1}{2}\frac{v}{\sqrt{1+\vert\nabla u\vert^2}} \left\llbracket \sigma\vert \nabla\psi\vert^2\right\rrbracket
\end{split} 
\end{equation*}
on $\Sigma$. From this we readily deduce that
\begin{equation}
\begin{split}\label{x4}
\left\llbracket \sigma\partial_s\psi(0)\nabla\psi+\frac{\sigma}{2}\vert \nabla\psi\vert^2\partial_s\Phi(0)\right\rrbracket\cdot {\bf n}_\Sigma =\frac{1}{2}\frac{v}{\sqrt{1+\vert\nabla u\vert^2}} \llbracket\sigma\rrbracket \, \nabla\psi_1\cdot \nabla\psi_2
\end{split} 
\end{equation}
on $\Sigma$. We finally derive an alternative expression for $\nabla\psi_1\cdot \nabla\psi_2$ on $\Sigma$. To this end note that we can express \eqref{TM1ax} in the form
\begin{equation}\label{A}
\nabla'\psi_1+\partial_z\psi_1\nabla u= \nabla'\psi_2+\partial_z\psi_2\nabla u \qquad\text{on}\ \ \Sigma\,,
\end{equation}
while \eqref{TM1} reads
\begin{equation}\label{B}
\sigma_1\partial_z\psi_1-\sigma_1 \nabla u\cdot \nabla'\psi_1= \sigma_2\partial_z\psi_2-\sigma_2 \nabla u\cdot \nabla'\psi_2 \qquad\text{on}\ \ \Sigma\,.
\end{equation}
Taking the inner product of \eqref{A} with $\sigma_1\nabla u$, adding \eqref{B} to the resulting identity, and multiplying the outcome by $\partial_z\psi_2$, one obtains 
\begin{equation}\label{z1}
\sigma_1\left(1+\vert\nabla u\vert^2\right)\partial_z\psi_1 \partial_z\psi_2=  \llbracket\sigma\rrbracket \partial_z\psi_2\nabla u\cdot\nabla'\psi_2 + \left(\sigma_1 \vert\nabla u\vert^2+\sigma_2\right)\vert\partial_z\psi_2\vert^2 \end{equation}
on $\Sigma$. Next, we take the inner product of \eqref{A} with $\sigma_1\nabla'\psi_2$, multiply \eqref{B} by $\nabla u\cdot\nabla'\psi_2$, and subtract the resulting identities. This yields
\begin{equation}\label{100}
\begin{split}
\sigma_1&\left( \nabla'\psi_1\cdot \nabla'\psi_2 +(\nabla u\cdot\nabla'\psi_1)( \nabla u\cdot\nabla'\psi_2)\right)\\
&\qquad\qquad = \sigma_1\vert\nabla'\psi_2\vert^2+\sigma_2 \left(\nabla u\cdot \nabla'\psi_2\right)^2+  \llbracket\sigma\rrbracket \partial_z\psi_2 \nabla u\cdot\nabla'\psi_2 
\end{split}
\end{equation}
on $\Sigma$. One then easily checks that
$$
(\nabla u\cdot\nabla'\psi_1)( \nabla u\cdot\nabla'\psi_2) =\vert\nabla u\vert^2  \nabla'\psi_1\cdot\nabla'\psi_2  -\big(\nabla'\psi_1\cdot \nabla^\perp u \big) \big(\nabla'\psi_2\cdot \nabla^\perp u \big)\,,
$$
where $\nabla^\perp:= (\partial_{x_2} ,-\partial_{x_1} )$. Since \eqref{A} implies $\nabla'\psi_1\cdot \nabla^\perp u  = \nabla'\psi_2\cdot \nabla^\perp u$,
we derive from \eqref{100} that
\begin{equation}\label{z2}
\begin{split}
\sigma_1 \left(1+ \vert\nabla u\vert^2\right) \nabla'\psi_1\cdot \nabla'\psi_2 & = \sigma_1\big(\nabla'\psi_2  \cdot \nabla^\perp u\big)^2 +
\sigma_1 \vert\nabla'\psi_2\vert^2 \\
& \quad +\sigma_2 \left(\nabla u\cdot \nabla'\psi_2\right)^2+  \llbracket\sigma\rrbracket \partial_z\psi_2 \nabla u\cdot\nabla'\psi_2 
\end{split}
\end{equation}
on $\Sigma$. Therefore, combining \eqref{z1} and \eqref{z2} we deduce that
\begin{equation}\label{z3}
\begin{split}
\sigma_1\left(1+ \vert\nabla u\vert^2\right) & \nabla\psi_1\cdot \nabla\psi_2 = \sigma_1\left(1+ \vert\nabla u\vert^2\right) \big( \nabla'\psi_1\cdot \nabla'\psi_2
+\partial_z\psi_1 \partial_z\psi_2\big)\\
& =   2 \llbracket\sigma\rrbracket \partial_z\psi_2\nabla u\cdot\nabla'\psi_2 + \left(\sigma_1 \vert\nabla u\vert^2+\sigma_2\right)\vert\partial_z\psi_2\vert^2 \\
& \qquad +
\sigma_1\big(\nabla'\psi_2  \cdot \nabla^\perp u\big)^2 +
\sigma_1 \vert\nabla'\psi_2\vert^2+\sigma_2 \left(\nabla u\cdot \nabla'\psi_2\right)^2\\
&= \sigma_1\left( \vert \partial_z\psi_2\nabla u +\nabla'\psi_2\vert^2+ \big(\nabla'\psi_2  \cdot \nabla^\perp u\big)^2\right)\\
&\qquad +\sigma_2\left( \partial_z\psi_2 -\nabla u\cdot\nabla'\psi_2\right)^2
\end{split}
\end{equation}
on $\Sigma$. Gathering \eqref{x4} and \eqref{z3} gives 
\begin{equation}\label{z4}
\begin{split}
\int_{\Sigma} & \left\llbracket \sigma\partial_s\psi(0)\nabla\psi+\frac{\sigma}{2}\vert \nabla\psi\vert^2\partial_s\Phi(0)\right\rrbracket\cdot {\bf n}_\Sigma\,\rd S\\
&\qquad= \frac{1}{2}\int_D \frac{v \llbracket \sigma\rrbracket}{1+ \vert\nabla u\vert^2} \Big\{ \vert \partial_z\psi_2[u]\nabla u +\nabla'\psi_2[u]\vert^2+ \big(\nabla'\psi_2[u]  \cdot \nabla^\perp u\big)^2\\
& \qquad\qquad\qquad\qquad\qquad\qquad\   +\frac{\sigma_2[u]}{\sigma_1}\left( \partial_z\psi_2[u] -\nabla u\cdot\nabla'\psi_2[u]\right)^2\Big\}\,\rd x\,.
\end{split}
\end{equation}
Finally observe that
\begin{equation}
\int_{\Omega_2} \partial_z\sh [z-u] \vert \nabla\psi_2\vert^2 v\,\rd (x,z) = \int_{D} v \int_{u}^{u+d} \partial_z\sh [z-u] \vert \nabla\psi_2\vert^2 \, \rd z \,\rd x\,. \label{z4.1}
\end{equation}
Consequently, we infer from \eqref{DE2}, \eqref{x2}, \eqref{z4}, and \eqref{z4.1} that the first variation of $E_e$ reads
\begin{equation}\label{FD}
\begin{split}
\delta_u E_e(u)=& \frac{1}{2} \int_{u}^{u+d} \partial_z\sh [z-u] \vert \nabla\psi_2[z]\vert^2 \, \rd z  +  \frac{1}{2}\sh [d]  \left\vert \nabla\psi_2 [u+d]\right\vert^2 \\
&-  \frac{1}{2}\,\frac{\llbracket \sigma\rrbracket}{1+ \vert\nabla u\vert^2} \left\{ \Big\vert \partial_z\psi_2[u]\nabla u +\nabla'\psi_2[u]\Big\vert^2+ \Big(\nabla'\psi_2[u]  \cdot \nabla^\perp u\Big)^2\right. \\
& \qquad\qquad\qquad\qquad\qquad\left. +\frac{\sh[0]}{\sigma_1}\Big( \partial_z\psi_2[u] -\nabla u\cdot\nabla'\psi_2[u]\Big)^2\right\} \,.
\end{split}
\end{equation}
Note that the second term of $\delta_u E_e(u)$ is always non-negative.  If the permittivity of the medium filling the region between the plates is smaller than the permittivity of the dielectric material the plate is made of, then the third term in $\delta_u E_e(u)$ is also non-negative since $\llbracket \sigma\rrbracket =\sigma_1-\sigma_2\le 0$. This is the case in many applications where the region between the plates is vacuumized. Finally, the first term of $\delta_u E_e(u)$ is nonlocal and need not have a constant sign but we emphasize that it vanishes if the permittivity of the plate is independent of the vertical direction as in \cite{Pe02}.
In fact, an interesting mathematical consequence of \eqref{FD} is that the electrostatic energy $E_e(u)$ is monotonically increasing with respect to the deflection $u$ as soon as $\sh[\cdot]$ is a non-decreasing function and $\sigma_1\le\sh$.

\begin{rem}\label{rem1d}
If $D$ is a one-dimensional interval, then formula \eqref{FD} for $\delta_u E_e(u)$ is still valid after setting $\nabla^\perp u := 0$.
\end{rem}

\begin{rem}\label{R4.1}
The electrostatic force $F_e(u)=\delta_u E_e(u)$ acting on the elastic plate found in \eqref{FD} markedly differs from the one taken in \cite{Pe02}, where the last term on the right-hand side of \eqref{FD}, involving the jump of the permittivity, is missing (the first term anyway does not come into play in \cite{Pe02} since no vertical variation in the permittivity is considered). The reason for this is that in the latter reference the electrostatic force is not derived as the first variation with respect to $u$ of the electrostatic energy  $E_e(u)$ as done above, but is assumed to be given \textit{a priori} by $\sigma_2 [u] \left\vert \nabla\psi_2 [u]\right\vert^2/2$ (corresponding to the second term in \eqref{FD}).
\end{rem}

\subsection{Mechanical Forces}

The mechanical energy $E_m ( {u})$ of the device includes three contributions. We first account for plate bending and external stretching by the terms
$$
\frac{B}{2}\int_{D}  \left\vert\Delta   u\right\vert^2\,\rd x + \frac{T}{2}\int_{  D}  \left\vert \nabla  u \right\vert^2\, \rd x\,,
$$
where $B$ is the product of Young's modulus with area moment of inertia of the cross section of the plate and $T$ is the coefficient of the axial tension force. We hence neglect nonlinear elasticity effects as well as internal stretching effects and take into account only vertical deflections. Finally, we model the natural fact that the upper plate cannot penetrate the ground plate by adding a constraint term which we choose to be
$$
\int_D \mathbb{I}_{[-H,\infty)} (u)\,\rd x\,.
$$
Here, $\mathbb{I}_{[-H,\infty)}$ denotes the indicator function of the closed interval $[-H,\infty)$ on which it takes the value zero and the value $\infty$ on its complement. 
Consequently, the mechanical energy reads
$$
E_m ( {u}):=\frac{B}{2}\int_{  D}  \left\vert\Delta   u \right\vert^2\,\rd x
+\frac{T}{2}\int_{  D}  \left\vert \nabla  u \right\vert^2\, \rd x +\int_D \mathbb{I}_{[-H,\infty)} (u)\,\rd x\,.
$$
It readily follows that the mechanical force is given by
\begin{equation}\label{obelix}
\delta_u E_m ( {u})=B \Delta^2   u - T \Delta u + \partial\mathbb{I}_{[-H,\infty)}(u)
\,,
\end{equation}
where $\partial\mathbb{I}_{[-H,\infty)}(u)$ is the subdifferential of the indicator function $\mathbb{I}_{[-H,\infty)}$. Recall that given $u\in L_2(D)$ satisfying $u\ge -H$ a.e. in $D$, a function $\zeta\in L_2(D)$ belongs to $\partial\mathbb{I}_{[-H,\infty)}(u)$ if and only if it satisfies the variational inequality
\begin{equation}\label{gutemine}
0\ge \int_D \zeta (v-u)\,\rd x \quad \text{ for all }\quad v\in L_2(D) \ \text{ with }\ v\ge -H\ \text{ a.e. in $D$}\,.
\end{equation}
Clearly, $\zeta\equiv 0$ if $u>-H$, that is, as long as the gap between the elastic plate and the ground plate is positive.

\begin{rem}\label{remPenalty}
The unilateral side condition $u\ge -H$ can also be modeled by a penalty term involving  the Heaviside function which amounts to replace $\partial\mathbb{I}_{[-H,\infty)}(u)$ in \eqref{obelix} by $- s\, {\rm Heav}(-H-u)$ with a sufficiently large number $s$.
\end{rem}

\subsection{Governing Equations for $(u,\psi_u)$}

To obtain now a complete model for the deflection $u$ and the electrostatic potential $\psi_u$, we include all forces and add a damping force. Thus, the evolution equation for the deflection $u$ reads
\begin{subequations}\label{umodel}
\begin{align}
\alpha_0\partial_t^2u+r\partial_t u &+  B  \Delta^2   u - T \Delta u +\zeta \nonumber\\
&= -\frac{1}{2} \int_{u}^{u+d} \partial_z\sh [z-u] \vert \nabla\psi_{u,2}[z]\vert^2 \, \rd z - \frac{1}{2}\sh[d]  \left\vert \nabla\psi_{u,2} [u+d]\right\vert^2 \nonumber \\
&\quad\ +  \frac{1}{2}\,\frac{\llbracket \sigma\rrbracket}{1+ \vert\nabla u\vert^2} \left\{ \Big\vert \partial_z\psi_{u,2}[u]\nabla u +\nabla'\psi_{u,2}[u]\Big\vert^2+ \Big(\nabla'\psi_{u,2}[u]  \cdot \nabla^\perp u\Big)^2\right. \label{eq} \\
& \qquad\qquad\qquad\qquad\qquad\left. +\frac{\sh[0]}{\sigma_1}\Big( \partial_z\psi_{u,2}[u] -\nabla u\cdot\nabla'\psi_{u,2}[u]\Big)^2\right\}\,, \nonumber
\end{align}
for $t>0$ and $x\in D$ with $\zeta(t)$ belonging to $\partial\mathbb{I}_{[-H,\infty)}(u(t))$ for $t>0$ (i.e. satisfying \eqref{gutemine}), supplemented with boundary conditions
\begin{equation}\label{ubca}
u= B\, \partial_\nu u=0\quad \text{on}\quad  \partial D\,, \quad t>0\,,
\end{equation}
\end{subequations}
and some initial conditions. The electrostatic potential $\psi_u$ satisfies 
\begin{subequations}\label{psimodel}
\begin{equation}\label{psisa}
\mathrm{div}\left( \sigma\nabla \psi_u\right)=0 \quad\text{in}\quad  {\Omega}({u})\,, \quad t>0\,,
\end{equation}
with \textit{transmission conditions} on the interface $\Sigma(u)$, 
\begin{equation}\label{TM1a}
 \llbracket \psi_u\rrbracket = \llbracket \sigma\nabla \psi_u \rrbracket \cdot {\bf n}_{ \Sigma(u)}= 0 \quad\text{on}\quad  \Sigma(u)\,, \quad t>0\,,
\end{equation}
 along with the boundary condition
\begin{equation}
 \psi_u=h_{u+d}\,,\quad  (x,z)\in \partial\Omega(u)\, \quad t>0\,.\label{psibc1a}
\end{equation} 
\end{subequations}
Recall that $h_{u+d}$ is defined in \eqref{bobbybrown}.

\section{The Thin Plate Limit $d\rightarrow 0$}\label{Sec3}

We next derive equations corresponding to \eqref{umodel}, \eqref{psimodel} in the limit $d\to 0$ of a thin elastic plate, the purpose of this derivation being twofold: besides the obvious goal of obtaining reduced models which are likely to be more tractable for theoretical and numerical investigations, we also aim at determining how the heterogeneity of the elastic plate -- reflected through the non-constant permittivity $\sigma_2$ -- impacts the limit models. The starting point is to identify the electrostatic potential in the limit $d\to 0$. Recall that, given a deflection $u$, the electrostatic potential $\psi_u$ satisfying \eqref{psimodel} (and thus depending on $d$) is a maximizer of the energy functional
$$
\mathcal{E}(u,\vartheta)=-\frac{\sigma_1}{2}\int_{\Omega_1(u)}\vert\nabla\vartheta\vert^2\,\rd (x,z)-\frac{1}{2}\int_D \int_{u}^{u+d} \sigma_2 \vert\nabla\vartheta\vert^2\,\rd z\rd x
$$
with respect to $\vartheta$ satisfying the boundary conditions \eqref{psibc1a}.
To find the limit of this functional as $d\to 0$,  we use $\Gamma$-convergence techniques  along the lines of \cite{AB86}. This then also allows us to derive the  corresponding force exerted on the thin elastic plate as in Section~\ref{Sec2.2}. 

Throughout this section we fix a smooth deflection
$u:\bar D\rightarrow (-H,M)$ with $M>0$ satisfying 
\begin{equation}\label{bcu}
u=\partial_\nu u=0\quad \text{on}\quad \partial D\,.
\end{equation}
Then
$$
 \Omega^d(u):=\Omega_1(u)\cup \Sigma(u)\cup \Omega_2^d(u) \subset \Omega_0:=D\times (-H,M+1)
$$
for $d\le 1$, where 
$$
 {\Omega}_2^d(u):=\left\{(x,z)\in D\times\mathbb{R}\,;\, u(x)<  z <  u(x)+d\right\}\,.
$$
We investigate two cases: first when the permittivity of the elastic plate is independent of $d$  and then when it scales with the plate thickness $d$.

\subsection{The Case $\sigma_2=O(1)$}\label{Sec3.1}

We here consider the case in which the dielectric profile of the elastic plate is of order~$1$ compared to the plate's thickness $d$. We thus assume that 
\begin{equation}
\sigma_2(x,z):= \sh(x,z-u(x))\,, \qquad (x,z)\in \Omega_2^d(u)\,, \label{O(1)}
\end{equation}
where $\sh$ is a continuous function on $\bar D\times [0,1]$ independent of $d$ and satisfying $\sh\ge \sigma_0>0$ for some constant $\sigma_0$. We set
\begin{equation*}
\begin{split}
G_d(u,\theta):=&\frac{\sigma_1}{2}\int_{\Omega_1(u)}\vert\nabla(\theta+h_{u+d})\vert^2\,\rd (x,z)\\
& + \frac{1}{2}\int_D\int_{u}^{u+d} \sh[z-u]\,\vert \nabla(\theta+h_{u+d})\vert^2\,\rd z\rd x\,,\quad \theta\in H_{0}^1(\Omega^d(u))\,,
\end{split}
\end{equation*}
and
$$
G_d(u,\theta):=\infty\,,\quad \theta\in L_2(\Omega_0)\setminus H_{0}^1(\Omega^d(u))\,.
$$
Moreover, we introduce
$$
G_0(u,\theta):=\frac{\sigma_1}{2}\int_{\Omega_1(u)}\vert\nabla(\theta+h_{u})\vert^2\,\rd (x,z)\,,\quad \theta\in H_{0}^1(\Omega_1(u))\,,
$$
and
$$
G_0(u,\theta):=\infty\,,\quad \theta\in L_2(\Omega_0)\setminus H_{0}^1(\Omega^d(u))\,.
$$

\subsubsection{Reduced Electrostatic Energy when $\sigma_2=O(1)$}\label{REE1}

The next result on $\Gamma$-convergence of the energies follows exactly as in \cite{AB86, BCF80}. We omit details here but refer to the next section for a similar computation in a more complicated situation. 

\begin{prop}\label{GC1}
Let $\sigma_2$ be given by \eqref{O(1)}. If $\theta\in H_{0}^1(\Omega_1(u))$, then 
$$
\Gamma-\lim_{d\to 0} G_d(u,\theta)  =G_0(u,\theta)\quad \text{in}\quad L_2(\Omega_0)\,.
$$
\end{prop}

Recalling the relation $\mathcal{E}(u,\vartheta)=-G_d(u,\vartheta-h_{u+d})$ for $\vartheta\in h_{u+d}+ H_{0}^1(\Omega^d(u))$ and $d>0$, it follows from Proposition~\ref{GC1} that in the limit $d\to 0$ the electrostatic energy $E_e(u)$ is given by
$$
E_e(u)=-\frac{\sigma_1}{2}\int_{\Omega_1(u)}\vert\nabla\psi_u\vert^2\,\rd (x,z)\,,
$$
where $\psi_u-h_{u}$ is a critical point of $G_0(u,\cdot)$ in $H_{0}^1(\Omega_1(u))$. Thus, the electrostatic potential $\psi_u$ solves the elliptic problem
\begin{subequations}\label{psimodel1}
\begin{equation}\label{psisa1}
\Delta \psi_u =0 \quad\text{in}\quad  \Omega_1({u})\,,
\end{equation}
 along with the boundary condition
\begin{equation}
 \psi_u=h_{u} \quad\text{on}\quad \partial\Omega_1(u)\,.\label{psibc1a1}
\end{equation}
\end{subequations}
One now argues as in Section~\ref{Sec2.2} to compute the electrostatic force which reads
$$
F_e(u)= \delta_u E_e(u)=\frac{\sigma_1}{2}\vert\nabla\psi_u[u]\vert^2\,.
$$

\begin{rem}
It is worth pointing out that the elastic force retains no effects of the dielectric properties of the elastic plate in the limit $d\rightarrow 0$ when $\sigma_2=O(1)$.
\end{rem}

\subsubsection{Reduced Model when $\sigma_2=O(1)$}

The mechanical forces being still given by \eqref{obelix}, we obtain from Section~\ref{REE1} that the reduced model for $(u,\psi_u)$ in the thin elastic plate limit $d\rightarrow 0$ reads
\begin{subequations}\label{umodel2}
\begin{equation}\label{eq2}
\begin{split}
\alpha_0 \partial_t^2u+r\partial_t u &+  B \Delta^2   u - T \Delta u = -\frac{\sigma_1}{2}\vert\nabla\psi_u[u]\vert^2\,,\quad x\in D\,,\quad t>0\,,
\end{split}
\end{equation}
supplemented with boundary conditions
\begin{equation}\label{ubca2}
u= B\, \partial_\nu u=0\quad \text{on}\quad  \partial D\,,
\end{equation}
\end{subequations}
and some initial conditions, and where the electrostatic potential $\psi_u$ satisfies \eqref{psimodel1}.

The above free boundary model \eqref{psimodel1}, \eqref{umodel2} is already well-known in the existing literature and is actually the building block in the modeling of MEMS when the  thickness of the elastic plate is neglected from the outset \cite{BGP00, FMP03, PeB03, PeT01}. Let us remark that, in \eqref{psimodel1}, \eqref{umodel2}, the electrostatic potential $\psi_u$ jumps from zero to $V$ at touchdown points $x\in D$ where $u(x)=-H$ according to the boundary condition \eqref{psibc1a}, see \eqref{bobbybrown}. Consequently, a touchdown of the elastic plate on the ground plate induces a singularity in the electrostatic force in that case. This also explains why the obstacle term vanishes. Questions regarding well-posedness and qualitative aspects of this model were investigated in \cite{ELW14, LW14a, LW16a}, an overview being provided in the survey \cite{LWBible}.

\subsection{The Case $\sigma_2=O(d)$}\label{RC}

We next consider the case in which the dielectric profile of the elastic plate scales with the plate's thickness $d$. This corresponds to a  highly-conducting material. More precisely, let 
\begin{equation}\label{O(d)}
\sigma_2(x,z)=d\, \sh(x,z-u(x))\,,\quad (x,z)\in \Omega_2^d(u)\,,
\end{equation}
for some continuous function $\sh$ on $\bar D\times [0,1]$ independent of $d$ satisfying $\sh\ge \sigma_0>0$ for some constant $\sigma_0$. We set
\begin{equation*}
\begin{split}
G_d(u,\theta):=&\frac{\sigma_1}{2}\int_{\Omega_1(u)}\vert\nabla(\theta+h_{u+d})\vert^2\,\rd (x,z)\\
& +
 \frac{d}{2}\int_D\int_{u}^{u+d} \sh[z-u]\,\vert \nabla(\theta+h_{u+d})\vert^2\,\rd z\rd x\,,\quad \theta\in H_{0}^1(\Omega^d(u))\,,
\end{split}
\end{equation*}
and
$$
G_d(u,\theta):=\infty\,,\quad \theta\in L_2(\Omega_0)\setminus H_{0}^1(\Omega^d(u))\,.
$$
Moreover, we define
$$
G_0(u,\theta):=\frac{\sigma_1}{2}\int_{\Omega_1(u)}\vert\nabla(\theta+h_{u})\vert^2\,\rd (x,z)+\frac{1}{2}\int_D  \sh[0]\,\theta[u]^2\, (1+\vert\nabla u\vert^2)\,\rd x
$$
for $\theta\in H_B^1(\Omega_1(u))$ and
$$
G_0(u,\theta):=\infty\,,\quad \theta\in L_2(\Omega_0)\setminus H_B^1(\Omega_1(u))\,,
$$
where
$$
H_B^1(\Omega_1(u)):=\{\theta\in H^1(\Omega_1(u))\,;\,\theta=0 \ \text{ on }\ \partial\Omega_1(u)\setminus \Sigma(u)\}\,.
$$
As shown below, $G_0(u,\cdot)$ turns out to be the $\Gamma$-limit of the functional $G_d(u,\cdot)$ as $d\to 0$. In contrast to the case $\sigma_2=O(1)$ previously studied in Section~\ref{Sec3.1}, the functional $G_0(u,\cdot)$ includes a term retaining the dielectric properties of the elastic plate.

\subsubsection{Reduced Electrostatic Energy when $\sigma_2=O(d)$}

We first identify the $\Gamma$-limit of the functional $G_d(u,\cdot)$ as $d\to 0$.

\begin{prop}\label{GC}
Let $\sigma_2$ be given by \eqref{O(d)}. If $\theta\in H_B^1(\Omega_1(u))$, then 
$$
\Gamma-\lim_{d\to 0} G_d(u,\theta)  =G_0(u,\theta)\quad \text{in}\quad L_2(\Omega_0)\,.
$$
\end{prop}

\begin{proof}
We follow the lines of \cite{AB86}, the main difference being that the domain $\Omega_2^d(u)$ is initially not para\-metrized  along the normal to $\Sigma(u)$. Since $u$ is a fixed smooth function satisfying \eqref{bcu} throughout the proof,  we omit for simplicity the dependence on $u$ in the notation of the functionals $G_d$, $G_0$, and the sets $\Omega_1$, $\Omega_2^d$, and $\Omega^d$. 

\medskip
\noindent {\bf Step 1: Asymptotic lower semi-continuity.}  It follows from \eqref{bcu} that there is $d_0>0$ such that for any $d\in (0,d_0)$, there is a smooth function $r_d: D\rightarrow (0,\infty)$ such that the mapping
$$
\Lambda: U^d\rightarrow \Omega_2^d\,,\quad (x,s)\mapsto \left(x- \frac{s\nabla u(x)}{\sqrt{1+\vert\nabla u(x)\vert^2}}\,,\, u(x)+ \frac{s}{\sqrt{1+\vert\nabla u(x)\vert^2}}\right)
$$
defines a $C^1$-diffeomorphism, where 
$$
U^d:=\{(x,s)\,;\, x\in D\,,\, 0<s<r_d(x)\}
$$
and 
\begin{equation}
\lim_{d\to 0} \|r_d\|_{L_\infty(D)} = 0 \,. \label{lost001}
\end{equation}
The determinant of its derivative is of the form
\begin{equation}\label{10a}
\mathrm{det}(D\Lambda(x,s))= \sqrt{1+\vert\nabla u(x)\vert^2} +s O(\|u\|_{W_\infty^2(D)})\,,\quad (x,s)\in U^d\,.
\end{equation}
Let $x\in D$ and $d\in (0,d_0)$. According to the definition of $r_d$, there is $y_d\in D$ such that $\Lambda(x,r_d(x))=(y_d,u(y_d)+d)$, that is,
\begin{equation}\label{9}
y_d=x- \frac{r_d(x)\nabla u(x)}{\sqrt{1+\vert\nabla u(x)\vert^2}}\,,\qquad  u(y_d)+d=u(x)+ \frac{r_d(x)}{\sqrt{1+\vert\nabla u(x)\vert^2}}\,,
\end{equation}
from which we obtain the implicit equation
\begin{equation*}
\frac{1}{-r_d(x)}\left[u\left(x- \frac{r_d(x)\nabla u(x)}{\sqrt{1+\vert\nabla u(x)\vert^2}}\right)-u(x)\right]=\frac{d}{r_d(x)}-\frac{1}{\sqrt{1+\vert\nabla u(x)\vert^2}}
\end{equation*}
for $r_d(x)$. Hence, by Taylor's expansion,
\begin{align*}
\left| \frac{d}{r_d(x)} - \sqrt{1+\vert\nabla u(x)\vert^2} \right| & = \left| \frac{d}{r_d(x)} - \frac{1}{\sqrt{1+\vert\nabla u(x)\vert^2}} - \frac{\vert\nabla u(x)\vert^2}{\sqrt{1+\vert\nabla u(x)\vert^2}} \right| \\
& \le |r_d(x)| \frac{|\nabla u(x)|^2}{1+|\nabla u(x)|^2} \|u\|_{W_\infty^2(D)} \\
& \le \|r_d\|_{L_\infty(D)} \|u\|_{W_\infty^2(D)}\,,
\end{align*}
so that, by \eqref{lost001},
\begin{equation}
\lim_{d\to 0}\, \left\| \frac{d}{r_d} - \sqrt{1+\vert\nabla u\vert^2} \right\|_{L_\infty(D)}  = 0\,.
\label{eq:1}
\end{equation}
Now, let $\theta_d\in H_{0}^1(\Omega^d)$ be such that $\theta_d\rightarrow \theta_0$ in $L_2(\Omega_0)$ as $d\to 0$. We claim that
\begin{equation}
G_0(\theta_0)\le\liminf_{d\to 0} G_d(\theta_d)\,.
\label{eq:2}
\end{equation}
First note that we may assume without loss of generality that $(G_d(\theta_d))_{d\le d_0}$ is bounded. Hence, owing to the definition of $G_d$ and the lower bound for $\sigma_*$, there is $c_0>0$ such that, for all $d\in (0,d_0)$,
\begin{equation}\label{8}
\frac{\sigma_1}{2}\int_{\Omega_1}\vert\nabla (\theta_d+h_{u+d})\vert^2\,\rd (x,z)+d\frac{\sigma_0}{2}\int_D\int_u^{u+d} \vert\nabla (\theta_d+h_{u+d})\vert^2\,\rd z\rd x \le c_0\,.
\end{equation}
In particular, since
\begin{equation}
|h_{u+d}(x,z) - h_u(x,z)| \le d \|D h\|_{L_\infty(\Omega_0\times (-H,M+1))}\,, \qquad (x,z)\in \Omega_2^d\,, \label{lost005}
\end{equation}
by \eqref{bobbybrown} we may assume further that $(\theta_d+h_{u+d})_{d\le d_0}$ converges weakly towards $\theta_0+h_u$ in $H^1(\Omega_1)$. This convergence implies not only that $\theta_0 \in H^1(\Omega_1)$ and satisfies
\begin{equation}
\frac{\sigma_1}{2}\int_{\Omega_1(u)}\vert\nabla(\theta_0+h_{u})\vert^2\,\rd (x,z) \le \liminf_{d\to 0} \frac{\sigma_1}{2}\int_{\Omega_1(u)}\vert\nabla(\theta_d+h_{u+d})\vert^2\,\rd (x,z)\,, \label{lost002}
\end{equation}
but also that
\begin{equation}
\theta_d \longrightarrow \theta_0 \quad\text{ in }\quad L_2(\partial\Omega_1)\,, \label{lost003}
\end{equation}
thanks to the compact embedding of $H^1(\Omega_1)$ in $L_2(\partial\Omega_1)$ and \eqref{lost005}. In particular, $\theta_0\in H_B^1(\Omega_1)$, so that $G_0(\theta_0)$ is finite. Furthermore, due to \eqref{lost002}, it suffices to show that
$$
\mathcal{G}_0(\theta_0)\le\liminf_{d\to 0} \mathcal{G}_d(\theta_d)\,,
$$
for the claim \eqref{eq:2} to be true, where
$$
\mathcal{G}_d(\theta):=\left\{\begin{array}{ll}
\displaystyle\frac{d}{2}\displaystyle\int_D\int_{u}^{u+d} \sh[z-u]\,\vert \nabla(\theta+h_{u+d})\vert^2\,\rd z\rd x\,,& \theta\in H_{0}^1(\Omega^d)\,,\\ \\
\infty\,,& \theta\in L_2(\Omega_0)\setminus H_{0}^1(\Omega^d)\,,
\end{array}\right.
$$
and
$$
\mathcal{G}_0(\theta):=\left\{\begin{array}{ll}
\displaystyle\frac{1}{2}\int_D  \sh[0]\,\theta[u]^2\, (1+\vert\nabla u\vert^2)\,\rd x\,,&\theta\in H_B^1(\Omega_1)\,,\\ \\
\infty\,,&\theta\in L_2(\Omega_0)\setminus H_B^1(\Omega_1)\,.
\end{array}\right.
$$
Since $\sh\in C(\bar{D}\times [0,1])$, there holds
$$
\lim_{d\to 0} \sup_{(x,z)\in \Omega_2^d} |\sh(x,z-u(x)) - \sh(x,0)| = 0\, 
$$
and we infer from \eqref{8} that
\begin{align*}
\liminf_{d\to 0} \mathcal{G}_d(\theta_d) & = \liminf_{d\to 0}\,\frac{d}{2}\int_D\int_{u}^{u+d} \sh[0]\,\vert \nabla(\theta_d+h_{u+d})\vert^2\,\rd z\rd x\\
& = \liminf_{d\to 0}\,\frac{d}{2} \int_D \mathcal{H}_d(x)\, \rd x\,,
\end{align*}
where
$$
\mathcal{H}_d(x) := \int_0^{r_d(x)} \left( \sh[0] \vert\nabla(\theta_d+h_{u+d})\vert^2 \right) \circ \Lambda(x,s) \vert \mathrm{det}(D\Lambda(x,s))\vert \,\rd s\,, \qquad x\in D\,.
$$
Next, for $x\in D$ and $d\in (0,d_0)$, the definition of $\Lambda$ and $r_d$ together with \eqref{bobbybrown} and the Cauchy-Schwarz inequality ensure that
\begin{align*}
\left| V-(\theta_d+h_{u+d})(x,u(x)) \right|^2 & = \left| (\theta_d+h_{u+d})(\Lambda(x,r_d(x)) - (\theta_d+h_{u+d})(\Lambda(x,0)) \right|^2 \\
& = \left| \int_0^{r_d(x)} \left( \nabla(\theta_d+h_{u+d}) \right)\circ \Lambda(x,s)) \partial_s \Lambda(x,s) \,\rd s \right|^2 \\
& \le \left( \int_0^{r_d(x)} \frac{|\partial_s\Lambda(x,s)|^2}{\sh[0]\circ \Lambda(x,s) \vert \mathrm{det}(D\Lambda(x,s))\vert} \,\rd s \right) \mathcal{H}_d(x) \\
& =  \frac{\mathcal{H}_d(x)}{\omega_d(x)}\,,
\end{align*}
with
$$
\omega_d(x) := \left( \int_0^{r_d(x)} \frac{1}{\sh[0]\circ \Lambda(x,s) \vert \mathrm{det}(D\Lambda(x,s))\vert} \,\rd s \right)^{-1}\,, \qquad x\in D\,.
$$
Therefore,
\begin{align}
\liminf_{d\to 0} \mathcal{G}_d(\theta_d) & = \liminf_{d\to 0}\,\frac{d}{2} \int_D \mathcal{H}_d(x)\, \rd x \nonumber \\
& \ge \liminf_{d\to 0} \frac{d}{2} \int_D \omega_d(x) \left| V-(\theta_d+h_{u+d})(x,u(x)) \right|^2\,\rd x\,, \label{lost004}
\end{align}
and we are left with identifying the last term of the right-hand side of \eqref{lost004}. To this end, we observe that \eqref{10a} and \eqref{eq:1} entail that
\begin{align*}
\lim_{d\to 0} \frac{1}{d\, \omega_d(x)} & = \lim_{d\to 0} \frac{r_d(x)}{d} \frac{1}{r_d(x)} \frac{1}{\omega_d(x)} = \frac{1}{\sh(x,0) (1+|\nabla u(x)|^2)}
\end{align*}
uniformly with respect to $x\in D$. Combining this convergence with \eqref{lost005} and \eqref{lost003} allows us to pass to the limit in the right-hand side of \eqref{lost004} and conclude that 
$$
\liminf_{d\to 0} \mathcal{G}_d(\theta_d) \ge \mathcal{G}_0(\theta_0)\,,
$$
after recalling that $h_u(x,u(x))=V$ for $x\in D$ by \eqref{bobbybrown}, whence \eqref{eq:2}.

\medskip

\noindent {\bf Step 2: Recovery sequence.} Let $\theta\in H_B^1(\Omega_1)$. We may extend $\theta$ in the $z$-direction so that $\theta$ belongs to $H^1(\Omega_0)$ as well. Defining
$$
\varphi_d(x,z):=\max\left\{0,1-\frac{(z-u(x))_+}{d}\right\}\,,
$$
we get $\theta_d:=\varphi_d\theta\in H_{0}^1(\Omega^d)$ and $\theta_d\rightarrow \theta$ in $L_2(\Omega_0)$. Then
\begin{equation*}
\begin{split}
\limsup_{d\to 0} \mathcal{G}_d(\theta_d) & =
 \limsup_{d\to 0}\,\frac{d}{2}\int_D\int_{u}^{u+d}\, \sh[z-u]\,\vert \varphi_d\nabla\theta+\nabla h_{u+d}+\theta\nabla\varphi_d\vert^2\, \rd z\rd x\\
& =
 \limsup_{d\to 0}\,\frac{d}{2}\int_D\int_{u}^{u+d}\, \sh[z-u]\,\vert \theta\vert^2\, \vert\nabla\varphi_d\vert^2\, \rd z\rd x\\
& =
 \limsup_{d\to 0}\,\frac{1}{2d}\int_D\int_{u}^{u+d}\, \sh[z-u]\,\vert \theta\vert^2\, \left(1+\vert\nabla u\vert^2\right)\, \rd z\rd x\,,
\end{split}
\end{equation*}
whence 
$$
\limsup_{d\to 0} \mathcal{G}_d(\theta_d)= \mathcal{G}_0(\theta)
$$
by \cite[Lemma~III.1]{AB86}. Since $\theta_d=\theta$ in $\Omega_1$, this readily implies that
$$
\limsup_{d\to 0} G_d(\theta_d)  = G_0(\theta)
$$
from which the assertion follows.
\end{proof}


\begin{rem} 
Proposition~\ref{GC} is likely to be true if condition~\eqref{bcu} is replaced by the weaker one $u=0$ on $\partial D$. In this case, however, the parametrization of the domain $\Omega_2^d(u)$  along the normal to $\Sigma(u)$ is more involved, and this is the difficulty to be overcome.
\end{rem}


\subsubsection{Reduced Model when $\sigma_2=O(d)$}\label{SecRs}

Using Proposition~\ref{GC} and arguing as  in Section~\ref{Sec3.1},  the electrostatic energy $E_e(u)$ for a given deflection $u$ reads in the limit $d\to 0$
\begin{equation}
E_e(u) := -\frac{\sigma_1}{2} \int_{\Omega_1(u)} |\nabla\psi_u|^2\,\rd (x,z) - \frac{1}{2} \int_D \sh[0]\, \big(\psi_u[u]-V\big)^2 \big(1+|\nabla u|^2\big)\,\rd x\ , \label{PhRe1}
\end{equation}
where $\psi_u - h_{u} \in H_B^1(\Omega_1(u))$ is a maximizer of 
$$
G_0(u,\theta)=\frac{\sigma_1}{2}\int_{\Omega_1(u)}\vert\nabla(\theta+h_{u})\vert^2\,\rd (x,z)+\frac{1}{2}\int_D  \sh[0]\,\theta[u]^2\, \big(1+\vert\nabla u\vert^2\big)\,\rd x
$$ in 
$$
H_B^1(\Omega_1(u))=\big\{\theta\in H^1(\Omega_1(u))\,;\,\theta=0 \ \text{ on }\ \partial\Omega_1(u)\setminus \Sigma(u)\big\}\,.
$$
Thus $\psi_u$ solves
\begin{subequations}\label{PhRe2}
\begin{equation}
\Delta \psi_u = 0\ , \qquad (x,z)\in \Omega_1(u)\ , \label{PhRe2a} 
\end{equation}
supplemented with the Dirichlet boundary conditions 
\begin{equation}
\psi_u = h_{u} \ , \qquad (x,z)\in \partial\Omega_1(u)\setminus\Sigma(u)\,,
\label{PhRe2b}
\end{equation}
and with mixed boundary conditions on $\Sigma(u)$
\begin{equation}
\sigma_1 \big( \partial_z \psi_u[u] - \nabla u\cdot \nabla'\psi_u[u] \big) + \sh[0] \big(1+\vert\nabla u\vert^2\big) \big(\psi_u[u]-V\big) = 0\ , \quad x\in D\ . \label{PhRe2c}
\end{equation}
\end{subequations}
We now compute the electrostatic force acting on the elastic plate which corresponds to the Fr\'echet derivative of the electrostatic energy $E_e(u)$ with respect to $u$. As in Section~\ref{Sec2.2} we consider $v\in C_0^\infty(D)$ and set $u_s:=u+s v$ for $s\in (-\sigma_0,\sigma_0)$, where $\sigma_0$ is chosen small enough such that $u_s>-H$ and the transformation 
$$
\Phi(s)(x,z) := \left( x , z + s \frac{H+z}{H+u(x)} v(x) \right)\ , \qquad (x,z)\in \Omega_1(u)\ , 
$$
is a $C^1$-diffeomorphism from $\Omega_1(u)$ onto $\Omega_1(u_s)$ for all $s\in (-\sigma_0,\sigma_0)$. We next define for $s\in (-\sigma_0,\sigma_0)$ the solution $\psi(s)$ to \eqref{PhRe2} with $u_s$ instead of $u$, that is, $\psi(s)$ solves
\begin{equation*}
\Delta \psi(s) = 0\ , \qquad (x,z)\in \Omega_1(u_s)\ , 
\end{equation*}
supplemented with Dirichlet boundary conditions on $\partial\Omega_1(u_s)\setminus\Sigma(u_s)$
\begin{equation*}
\psi(s) =  h_{u_s}\ , \qquad (x,z)\in [D\times \{-H\}] \cup [\partial D\times (-H,0)]\ , 
\end{equation*}
and with mixed boundary conditions on $\Sigma(u_s)$
\begin{equation*}
\sigma_1 \big( \partial_z \psi(s)[u_s] - \nabla u_s\cdot \nabla'\psi(s)[u_s] \big) + \sh[0] \big(1+\vert\nabla u_s\vert^2\big) \big(\psi(s)[u_s]-V\big) = 0
\end{equation*}
for $x\in D$. Then, for $s\in (-\sigma_0,\sigma_0)$,
$$
E_e(u_s) = -\frac{\sigma_1}{2} \int_{\Omega_1(u_s)} |\nabla\psi(s)|^2\,\rd (x,z) - \frac{1}{2} \int_D \sh[0]\, \big(\psi(s)[u_s]-V\big)^2 \big(1+|\nabla u_s|^2\big)\,\rd x\ ,
$$
and, using again the Reynolds transport theorem, Gauss' theorem, and \eqref{PhRe2a}
\begin{align*}
\frac{\rd}{\rd s} E_e(u_s)|_{s=0} & = -\sigma_1 \int_{\Omega_1(u)} \left[ \nabla \psi_u \cdot \nabla\partial_s\psi(0) + \mathrm{div}\left( \frac{|\nabla\psi_u |^2}{2} \partial_s\Phi(0) \right) \right]\, \rd(x,z) \\
& \quad - \int_D \sh[0] \big(\psi_u[u]-V\big)^2 \nabla u\cdot \nabla v\,\rd x \\
& \quad - \int_D \sh[0] \big(1+|\nabla u|^2\big) \big(\psi_u[u]-V\big) \big( \partial_s\psi(0)[u] + \partial_z\psi_u[u] v \big)\,\rd x \\
& = -\sigma_1 \int_{\partial \Omega_1(u)} \left[ \partial_s\psi(0) \nabla\psi_u\cdot \mathbf{n}_{\partial\Omega_1(u)} + \frac{|\nabla\psi_u |^2}{2} \partial_s\Phi(0) \cdot \mathbf{n}_{\partial\Omega_1(u)} \right]\,\rd S \\
& \quad - \int_D \sh[0] \big(\psi_u[u]-V\big)^2 \nabla u\cdot \nabla v\,\rd x \\
& \quad - \int_D \sh[0] \big(1+|\nabla u|^2\big) \big(\psi_u[u]-V\big) \big( \partial_s\psi(0)[u] + \partial_z\psi_u[u] v \big)\,\rd x\ .
\end{align*}
Since 
$$
\partial_s\Phi(0)(x,z) = \left( 0, \frac{H+z}{H+u(x)} v(x) \right)\ , \qquad (x,z)\in\Omega_1(u)\ , 
$$
and $\partial_s\psi(0)=0$ on $D\times\{-H\}$ and $\partial D\times (-H,0)$  by \eqref{bobbybrown} and \eqref{bcu}, we further obtain
\begin{align*}
\frac{\rd}{\rd s} E_e(u_s)|_{s=0} & = -\sigma_1 \int_D \left( \partial_z \psi_u[u] - \nabla u\cdot \nabla'\psi_u[u] \right) \partial_s\psi(0)[u]\,\rd x \\
& \quad - \frac{\sigma_1}{2} \int_D |\nabla\psi_u[u]|^2 v \,\rd x + \int_D \mathrm{div}\left( \sh[0] \big(\psi_u[u]-V\big)^2 \nabla u \right) v \,\rd x \\
& \quad - \int_D \sh[0] \big(1+|\nabla u|^2\big) \big(\psi_u[u]-V\big) \big( \partial_s\psi(0)[u] + \partial_z\psi_u[u] v \big)\,\rd x\ .
\end{align*}
Owing to \eqref{PhRe2c}, the contributions involving $\partial_s\psi(0)[u]$ cancel and we end up with
\begin{align*}
\frac{\rd}{\rd s} E_e(u_s)|_{s=0}  = & -\frac{\sigma_1}{2} \int_D |\nabla\psi_u[u]|^2 v\,\rd x + \int_D \mathrm{div}\left( \sh[0] \big(\psi_u[u]-V\big)^2 \nabla u \right) v \,\rd x \nonumber \\
& - \int_D \sh[0] \big(1+|\nabla u|^2\big) \big(\psi_u[u]-V\big) \partial_z\psi_u[u] v \,\rd x\ ,
\end{align*}
so that the electrostatic force exerted on the plate is 
\begin{subequations}\label{PhR}
\begin{equation}
\begin{split}
F_e(u)  :=  &-\frac{\sigma_1}{2} |\nabla\psi_u[u]|^2 - \sh[0] \big(1+|\nabla u|^2\big) \big(\psi_u[u]-V\big) \partial_z\psi_u[u] \\
&   + \mathrm{div}\left( \sh[0] \big(\psi_u[u]-V\big)^2 \nabla u \right)\ .
\end{split} \label{PhRe4}
\end{equation}
Consequently, when $\sigma_2=O(d)$, the evolution of $u$ is given in the thin elastic plate limit $d\to 0$ by 
\begin{equation}\label{PhRe5}
\alpha_0 \partial_t^2 u + r \partial_t u +  B  {\Delta}^2   u - T \Delta u +\zeta = -F_e(u)\ , \qquad x\in D\,,  \quad t>0\,,
\end{equation}
supplemented with clamped boundary conditions 
\begin{equation}
u = B \partial_\nu u = 0\ , \qquad x\in\partial D\,,\quad t>0\,, \label{PhRe5b}
\end{equation} 
\end{subequations}
with $\zeta(t)$ belonging to $\partial\mathbb{I}_{[-H,\infty)}(u(t))$ for $t>0$
and $\psi_u$ solving \eqref{PhRe2}.

\section{Vanishing Aspect Ratio Limit $\ve\to 0$}\label{Sec4}

The previously presented models are complex in that they couple an evolution equation for the deflection $u$ involving a nonlinear nonlocal source term depending on the electrostatic potential $\psi_u$ to an elliptic boundary value problem for the latter on a domain moving according to the evolution of $u$. It is thus worth looking for simpler and more tractable models in order to get a better insight into the dynamics. A well-documented  simplification is the so-called vanishing aspect ratio limit which allows one to express the electrostatic potential $\psi_u$ explicitly in terms of the deflection $u$ and gives rise to models featuring a single equation for $u$ with only a local source term \cite{Pe02, BGP00, EGG10, PeB03, PeT01}. In this limit the vertical extent of the MEMS device is assumed to be much smaller than its horizontal dimension. A prior step is to properly rescale the variables and the unknowns. For simplicity we assume throughout this section that the function $h$ introduced in \eqref{bobbybrown} is explicitly given by 
$$
h(x,z,w) := \frac{V(H+z)}{H+w}\,,\quad (x,z,w)\in D \times (-H,\infty)\times (-H,\infty)\,.
$$

\subsection{Rescaled Equations for the Transmission Model \eqref{umodel}, \eqref{psimodel}}

We introduce dimensionless variables in  equations \eqref{psimodel} for $\psi_u$ and \eqref{umodel} for $u$. More precisely, we scale variables according to
$$
\tilde t := \frac{t}{rL^4}\,, \quad \tilde x:= \frac{ x}{L}\,,\quad  \tilde z:= \frac{z}{H}\,, \quad \tilde u:= \frac{ u}{H}\,, \quad \tilde{\psi}_{\tilde{u},\ell}:= \frac{\psi_{u,\ell}}{V}\,, \quad \tilde \sigma := \frac{\sigma}{\sigma_1}\,, \quad \tilde{\sigma}_* := \frac{\sh}{\sigma_1}\,,
$$
and define the relative thickness $\delta:=d/H$ of the elastic plate and the aspect ratio $\ve:=H/L$ of the device.
Accordingly, we introduce $\tilde D:=\{\tilde x \in\mathbb{R}^2\,; L\tilde x \in D\}$,
$$
\tilde \Omega_1(\tilde u):=\left\{(\tilde x,\tilde z)\in \tilde{D}\times \mathbb{R}\,;\,\, -1<\tilde z<\tilde u(\tilde x)\right\}
$$
and 
$$
\tilde \Omega_2 (\tilde u):=\left\{(\tilde x, \tilde z)\in \tilde D\times \mathbb{R}\,;\,\, \tilde u(\tilde x ) < \tilde z < \tilde u(\tilde x )+\delta\right\}\,
$$
with interface
$$
\tilde \Sigma(\tilde u):=\{(\tilde x, \tilde z) \in \tilde D \times\mathbb{R}\,;\,\,\, \tilde z=\tilde u(\tilde x) \}\,.
$$
We then use these relations in \eqref{umodel} and \eqref{psimodel}  to derive dimensionless equations. Dropping the tilde everywhere, we get for the dimensionless electrostatic potential
\begin{subequations}\label{wyzardpsi}
\begin{align}
\ve^2\mathrm{div}'\left( \sigma\nabla' {\psi_u}\right)+\partial_z(\sigma \partial_z\psi_u)&=0 &&\text{in}\quad  \Omega(u)\,,\label{psiscaled}\\
\llbracket \psi_u\rrbracket= \ve^2\left\llbracket \sigma\nabla'\psi_u\right\rrbracket\cdot \nabla u-\left\llbracket \sigma\partial_z\psi_u\right\rrbracket&= 0  &&\text{on}\quad  \Sigma(u)\,, \label{TM2scaled}\\
 \psi_u&= b_{u+\delta}\,,&&  \text{on}\quad \partial\Omega(u)\,,\label{psibc1scaled}
\end{align}
\end{subequations}
where
$$
\Omega(u)=\left\{(x,z)\in D\times \mathbb{R}\,;\,\, -1<  z< u(x)+\delta\right\}
$$
and
$$
b_{u+\delta}(x,z):=\frac{1+ z}{1+u(x)+\delta}\,,\quad (x,z)\in \Omega(u)\,.
$$
Also, we obtain for the dimensionless deflection of the elastic plate the evolution equation
\begin{subequations}\label{wyzardu}
\begin{equation}\label{eqscaled}
\begin{split}
\gamma^2\partial_t^2u+ &\partial_t u + \beta \Delta^2   u - \tau \Delta u +\zeta =-\lambda g_{\delta,\varepsilon}(u)\,,\qquad x\in D\,,\quad t>0\,,
\end{split}
\end{equation}
with $\zeta(t)$ belonging to $\partial\mathbb{I}_{[-1,\infty)}(u(t))$ for $t>0$ and subject to the boundary conditions
\begin{equation}\label{ubc}
u=\beta\partial_\nu  u=0\,,\quad x\in\partial  D\,,\quad t>0\,,
\end{equation}
where
\begin{align}
g_{\delta,\varepsilon}(u):= & \frac{1}{2}\int_{u}^{u+\delta} \partial_z \sh[z-u] \left( \varepsilon^2 |\nabla'\psi_{u,2}[z]|^2 + |\partial_z\psi_{u,2}[z]|^2 \right) \,\rd z \nonumber \\
& + \frac{1}{2}\sh[\delta] \left(\ve^2\left\vert \nabla'\psi_{u,2} [u+\delta]\right\vert^2+\left(\partial_z \psi_{u,2} [u+\delta]\right)^2\right) \nonumber\\
& + \frac{1}{2}\frac{\sh[0] -1}{1+ \ve^2\vert\nabla u\vert^2} \left\{ \ve^2\Big\vert \partial_z\psi_{u,2}[u]\nabla u +\nabla'\psi_{u,2}[u]\Big\vert^2+ \ve^4 \Big(\nabla'\psi_{u,2}[u]  \cdot \nabla^\perp u\Big)^2\right\} \label{gg} \\
& + \frac{1}{2}\frac{(\sh[0] -1) \sh[0]}{1+ \ve^2\vert\nabla u\vert^2} \Big( \partial_z\psi_{u,2}[u] -\ve^2\nabla u\cdot\nabla'\psi_{u,2}[u]\Big)^2 \,, \nonumber
\end{align}
\end{subequations}
and
$$
\gamma^2:= \frac{\alpha_0}{r^2 L^{4}}\,,\qquad \beta:=B \,,\qquad \tau:=TL^2\,,\qquad \lambda=\lambda(\ve):=\frac{\sigma_1V^2L}{\ve^3}\ .
$$
The rescaled total energy for a given deflection $u$ is 
$$
E(u):=E_m(u)+\lambda E_e(u)
$$
with rescaled mechanical energy
\begin{equation}\label{Emscale}
 E_m ( {u})=\frac{\beta}{2} \int_D\vert\Delta   u\vert^2\,\rd x + \frac{\tau}{2} \int_D\vert\nabla u\vert^2\,\rd x + \int_D\mathbb{I}_{[-1,\infty)}(u)\,\rd x
\end{equation}
and electrostatic energy $\lambda E_e(u)$, where
\begin{equation}\label{Eescale}
E_e(u)=- \frac{1}{2} \int_{ \Omega(u)} \sigma \left(\ve^2\vert\nabla'\psi_u\vert^2 +(\partial_z\psi_u)^2\right)\,\rd (x,z)\,.
\end{equation}
Note that we single out the dependence of the total energy $E$ on the parameter $\lambda$  as the dynamics of the model is very sensitive to the tuning of this parameter. 

\subsection{Vanishing Aspect Ratio Limit for the Transmission Model \eqref{umodel}, \eqref{psimodel}}

We next derive a simplified model from \eqref{umodel}, \eqref{psimodel} by letting the aspect ratio $\ve=H/L$ tend to zero while keeping $\delta=d/H>0$ fixed.
Setting $\ve=0$ in \eqref{wyzardpsi}, it readily follows from \eqref{psiscaled} and \eqref{TM2scaled} that there is a function $A$ independent of $z$ such that
$$
\partial_z\psi_{u,1} (x,z)= A(x)\,,\quad (x,z)\in \Omega_1(u)\,,
$$
and
$$
\sh(x,z-u(x)) \partial_z\psi_{u,2} (x,z)= A(x)\,,\quad (x,z)\in \Omega_2(u)\,.
$$
We then integrate the above equations in $z$ and use the boundary conditions \eqref{psibc1scaled} to obtain
$$
\psi_{u,1} (x,z)= A(x)(1+z)\,,\quad (x,z)\in \Omega_1(u)\,,
$$
and
$$
\psi_{u,2} (x,z)= 1-A(x)\int_z^{u(x)+\delta} \frac{\rd q}{\sh(x,q-u(x))}\,,\quad (x,z)\in \Omega_2(u)\,.
$$
Since $\psi_{u,1}$ and $\psi_{u,2}$ coincide along $\Sigma(u)$ by \eqref{TM2scaled}, we can compute $A$ as
$$
A(x)=\left(1+u(x)+N_\delta(x)\right)^{-1}\,,\quad x\in D\,,
$$
where
$$
N_\delta(x):=\int_{0}^{\delta} \frac{\rd q}{\sh(x,q)}\,,\quad x\in D\,.
$$
We then deduce that
$$
\partial_z\psi_{u,2} (x,z)= \left[ \sh(x,z-u(x))\left(1+u(x)+N_\delta(x)\right) \right]^{-1}\,,\quad (x,z)\in \Omega_2(u)\,.
$$
Setting $\ve=0$ and using the above formula in \eqref{gg},
 the force exerted on the elastic plate is given by
\begin{equation}\label{gsmg}
g_{\delta,0}(u)(x)= \frac{1}{2}\left(1+u(x)+N_\delta(x)\right)^{-2}\,,\quad x\in D\,.
\end{equation}
Let us point out that the electrostatic energy $E_e(u)$ is then
\begin{equation}
E_e(u)=-\frac{1}{2} \int_D\frac{\rd x}{1+u(x)+N_\delta(x)} \label{lost006}
\end{equation}
and thus coincides with the one from \cite[Section 4.4]{AmEtal}. 
Recalling \eqref{wyzardu} we end up with a single equation for the deflection $u$ which reads
\begin{equation}\label{eqqq}
\begin{split}
\gamma^2\partial_t^2u+\partial_t u &+  \beta  \Delta^2   u - \tau \Delta u +\zeta=  -\frac{\lambda}{2 \left(1+u+N_\delta\right)^{2}} \,,\qquad x\in D\,,\quad t>0\,,
\end{split}
\end{equation}
subject to the boundary condition \eqref{ubc}. In \eqref{eqqq}, the function  $\zeta(t) \in\partial \mathbb{I}_{[-1,\infty)}(u(t))$ accounts for the constraint  $u\ge -1$.
Equation \eqref{eqqq} as well as the electrostatic energy $E_e(u)$ depend weakly on dielectric properties of the top plate since $N_\delta\to 0$ as $\delta\to 0$. Hence no such effects are retained in this limit. This is consistent with our findings in Section~\ref{Sec3.1}.

In fact, in the limit $\delta=0$ of zero thickness, equation~\eqref{eqqq} reduces to the commonly used vanishing aspect ratio equation
\begin{equation}\label{eqqqq}
\begin{split}
\gamma^2\partial_t^2u+\partial_t u &+  \beta  \Delta^2   u - \tau \Delta u  =  -\frac{\lambda}{2 \left(1+u\right)^{2}} \,,\qquad x\in D\,,\quad t>0\,.
\end{split}
\end{equation}
 
\begin{rem} 
If the dielectric $\sigma_2$ of the plate is independent of the vertical coordinate, then the electrostatic potential computed above coincides with the one from \cite[Equations~(2.15), (2.17)]{Pe02}. However, the final form $-\lambda/2(1+u)^2$ of the electrostatic force in equation \eqref{eqqqq} does not include any dielectric effects and therefore differs markedly from \cite[Equation~(2.19)]{Pe02} which features prominently such effects. The reason for this discrepancy is that the electrostatic force considered in \cite{EGG10, Pe02} does not correspond to the one from \eqref{eq} derived from the electrostatic energy functional  (see also Remark~\ref{R4.1}).
\end{rem}

The small gap equation~\eqref{eqqqq} for a thin plate has been thoroughly investigated in the last two decades, see e.g. \cite{EGG10, FMPS06, Flo14, GPW05, LWBible, BGP00, LiL12, KLNT11} and the references therein. 
This equation features a singularity when $u$ approaches the value $-1$ which has the following consequences: on the one hand, there is no stationary solution when $\lambda$ exceeds a certain threshold value. On the other hand, if $\lambda$ is sufficiently large, then the solution to the evolution problem does not exist for all times and ceases to exist when $u$ reaches the value~$-1$ at a certain time.

A striking difference between equation~\eqref{eqqqq} for zero thickness  $\delta=0$ and equation~\eqref{eqqq} for positive thickness $\delta>0$ is that no such singularity occurs in the latter due to the constraint $u\ge -1$ (provided $N_\delta>0$, of course, which is the case when the plate is a dielectric material everywhere). Nevertheless, the touchdown phenomenon may still take place, but corresponds to a so-called \textit{zipped state} \cite{GB01} in which the constraint is saturated, meaning that the set of points in $D$ at which $u$ takes the value $-1$ is not empty. Equivalently, $\zeta\not\equiv 0$ in \eqref{eqqq}. Nonetheless, the dynamics of $u$ is then still governed by an evolution equation and there is no model breakdown. Equation~\eqref{eqqq} was also derived in \cite[Equation~(17)]{GB01} in a different set-up, where a layer of insulating material with constant dielectric and thickness $d$ is on top of the ground plate. Zipped states were investigated numerically therein. We also refer to \cite{AmEtal} for other related models.

\subsection{Vanishing Aspect Ratio Limit for the Highly-Conducting Model \eqref{PhRe2}, \eqref{PhR}}\label{SecRV}

To study the limiting behavior in \eqref{PhRe2}, \eqref{PhR} when the aspect ratio $\varepsilon= H/L$ of the device vanishes we scale variables and unknowns as
$$
\tilde{t} := \frac{t}{r L^4}\ , \quad \tilde{x} := \frac{x}{L}\ , \quad \tilde{z} := \frac{z}{H}\ , \quad \tilde{u} := \frac{u}{H}\ , \quad \tilde{\psi}_u := \frac{\psi_u}{V}\ , \quad \tilde{\sigma}_* := \frac{H \sh}{\sigma_1}\ .
$$
Introducing $\tilde{D} := \{\tilde{x}\in\mathbb{R}^2\ ;\ L\tilde{x}\in D\}$ and
$$
\tilde{\Omega}_1(\tilde{u}) := \{ (\tilde{x},\tilde{z})\in \tilde{D}\times \mathbb{R}\ ;\ -1 < \tilde{z} < \tilde{u}(\tilde{x}) \}\ ,
$$
it follows from \eqref{PhRe2} and \eqref{PhR} that, after dropping the tilde, $\psi_u$ solves the rescaled Laplace equation
\begin{subequations}\label{PhRe6}
\begin{equation}
\varepsilon^2 \Delta' \psi_u + \partial_z^2 \psi_u = 0\ , \qquad (x,z)\in \Omega_1(u)\ , \label{PhRe6a} 
\end{equation}
supplemented with Dirichlet boundary conditions on $\partial\Omega_1(u)\setminus\Sigma(u)$
\begin{equation}
\psi_u(x,z) = \frac{1+z}{1+u(x)}\ , \qquad (x,z)\in [D\times \{-1\}] \cup [\partial D\times (-1,0)]\ , \label{PhRe6b}
\end{equation}
and with mixed boundary conditions on $\Sigma(u)$
\begin{equation}
\partial_z \psi_u[u] - \varepsilon^2 \nabla u\cdot \nabla'\psi_u[u] + \sh[0] (1+\varepsilon^2 \vert\nabla u\vert^2) (\psi_u[u]-1) = 0\ , \qquad x\in D\ , \label{PhRe6c}
\end{equation}
\end{subequations}
while the evolution of $u$ is given by
 \begin{subequations}\label{PhRe7}
\begin{equation}
\gamma^2 \partial_t^2 u + \partial_t u +  \beta \Delta^2   u - \tau \Delta u + \zeta = - \lambda g_\ve(u)\ , \qquad x\in D\,,\quad t>0\, , \label{PhRe7a}
\end{equation}
with $\zeta(t) \in \partial\mathbb{I}_{[-1,\infty)}(u(t))$, supplemented with clamped boundary conditions 
\begin{equation}
u = \beta \partial_\nu u = 0\ ,\qquad x\in\partial D\,\quad t>0\,, \label{PhRe7b}
\end{equation} 
and the electrostatic force $-g_\ve(u)$ reads 
\begin{align}
g_\ve(u) & := - \frac{\varepsilon^2}{2} |\nabla'\psi_u[u]|^2 - \frac{1}{2} |\partial_z \psi_u[u]|^2 \nonumber\\
& \qquad - \sh[0] (1+\varepsilon^2 |\nabla u|^2) (\psi_u[u]-1) \partial_z\psi_u[u] \label{PhRe8} \\
& \qquad + \varepsilon^2 \mathrm{div}\left( \sh[0] (\psi_u[u]-1)^2 \nabla u \right)\ .  \nonumber
\end{align}
\end{subequations}
The parameters $\gamma$, $\beta$, $\tau$, and $\lambda$ in \eqref{PhRe7a} are given by
$$
\gamma^2 := \frac{\alpha_0}{r L^4} \ , \quad \beta := B\ , \quad \tau := TL^2\ , \quad \lambda := \frac{\sigma_1 V^2 L}{\varepsilon^3}\, .
$$
Let us now identify the vanishing aspect ratio limit $\varepsilon\to 0$ of \eqref{PhRe6}-\eqref{PhRe7}. We first infer from \eqref{PhRe6a} and \eqref{PhRe6c} that
$$
\partial_z \psi_u(x,z) = \sh(x,0) (1-\psi_u(x,u(x))\ , \qquad (x,z)\in \Omega_1(u)\ ,
$$
hence, taking into account that $\psi_u(x,-1)=0$ for $x\in D$ by \eqref{PhRe6b}, 
$$
\psi_u(x,z) = \sh(x,0) (1-\psi_u(x,u(x)) (z+1)\ , \qquad (x,z)\in \Omega_1(u)\ .
$$
In particular, taking $z=u(x)$, $x\in D$, in the previous identity gives 
$$
\psi_u(x,u(x)) = \frac{\sh(x,0)(1+u(x))}{1 + \sh(x,0)(1+u(x))}\ , \qquad x\in D\ ,
$$
and thus
\begin{equation*}
\psi_u(x,z) = \frac{\sh(x,0) (1+z)}{1 + \sh(x,0) (1+u(x))}\ , \qquad (x,z)\in \Omega_1(u)\ .
\end{equation*}
We next set $\varepsilon=0$ in \eqref{PhRe8} and find that
\begin{equation*}
g_0(u)=  -\frac{1}{2} |\partial_z \psi_u[u]|^2 - \sh[0] (\psi_u[u]-1) \partial_z \psi_u[u] = \frac{1}{2} \left( \frac{\sh[0]}{1+\sh[0](1+u)} \right)^2 
\end{equation*}
for $x\in D$ and $t>0$. Hence we obtain the governing equation for $u$ in the form
\begin{equation}\label{stanleyclarke}
\gamma^2\partial_t^2u+\partial_t u +  \beta  \Delta^2   u - \tau \Delta u +\zeta=  - \frac{\lambda}{2} \left( \frac{1}{1+u+\sh[0]^{-1}} \right)^2
\end{equation}
for $x\in D$ and $t>0$, with $\zeta(t) \in \partial\mathbb{I}_{[-1,\infty)}(u(t))$, supplemented with clamped boundary conditions \eqref{PhRe7b}.

Let us point out that equation~\eqref{stanleyclarke} is similar to equation~\eqref{eqqq} with $N_\delta$ replaced by $1/\sh[0]$, so that the vanishing aspect ratio limits as $\ve\to 0$ of the transmission model \eqref{umodel}, \eqref{psimodel} and of the highly-conducting model \eqref{PhRe2}, \eqref{PhR} give rise to similar equations.

\section{Discussion}

In the present paper we derived models for a MEMS device which take into account the thickness of the elastic plate and its dielectric properties. Our approach relies on the computation of the electrostatic force exerted on the elastic plate as the first variation of the electrostatic energy and thus contrasts with the derivation of related models in the existing literature. The resulting force differs from those of previous works in that it involves additional terms accounting for the jump of the permittivity across the device. Our models also incorporate a constraint  accounting for the fact that the elastic plate cannot penetrate the ground plate. An interesting feature of this constraint is that it prevents the breakdown of the models when pull-in occurs. Alternative models in this direction are proposed in \cite{GB01, LE08, LLG14, LLG15}. We next focused on the derivation of models with reduced complexity. We first considered the case when the thickness $d$ of the elastic plate vanishes. If the dielectric permittivity is of order one with respect to $d$, then the reduced model obtained in the limit $d\to 0$ does not retain any effects of the dielectric. However, such effects still play a role in the reduced model obtained in the limit $d\to 0$ if the dielectric permittivity is of order $d$. We finally performed the classical vanishing aspect ratio limit when the vertical dimension is much smaller compared to the horizontal ones. We obtained an explicit formula for the electrostatic force in terms of the deflection. The final model then only involves a single equation for the deflection and shows different features than corresponding models in the existing literature: the source term is well-defined as long as the permittivity does not vanish and is only singular on the zero set of the permittivity. From a mathematical viewpoint, several questions arise from the previous analysis: besides the well-posedness of the transmission problem \eqref{wyzardpsi}, \eqref{wyzardu} and the highly-conducting model \eqref{PhRe6}, \eqref{PhRe7}, it is also worth investigating the dynamics of the vanishing ratio models \eqref{eqqq} and \eqref{stanleyclarke}, including the existence of zipped stationary states and their possible multiplicity. We plan to investigate further these issues in future works.



\end{document}